\documentclass[12pt,reqno]{amsart}

\usepackage{amsmath}
\usepackage{amsthm}
\usepackage{amsopn}
\usepackage{amssymb}
\usepackage{enumerate}

\usepackage{latexsym}
\usepackage{verbatim} 

\usepackage[mathscr]{eucal}

\newtheorem{thm}{Theorem}[section]

\newtheorem*{3k-4}{Freiman's $3k-4$ Theorem}
\newtheorem{lemma}[thm]{Lemma}

\newtheorem{conj}[thm]{Conjecture}
\newtheorem{prop}[thm]{Proposition}

\newtheorem{cclaim}{Claim}
\newtheorem*{claim*}{Claim}
\newtheorem*{claim3}{Claim 3$^{\prime}$}

\theoremstyle{definition}

\newtheorem*{defn}{Definition}

\newcommand{\ds}{\displaystyle}

\newcommand{\SF}{\mathrm{SF}}

\def\B{\mathcal{B}}

\def\HH{\mathcal{H}}
\def\I{\mathcal{I}}

\def\S{\mathcal{S}}

\def\X{\mathcal{X}}

\def\N{\mathbb{N}}
\def\Pr{\mathbb{P}}

\def\ZZ{\mathbb{Z}}

\def\a{\mathbf{a}}

\def\le{\leqslant}
\def\ge{\geqslant}

\def\eps{\varepsilon}
\def\<{\langle}
\def\>{\rangle}

\def\span{\textup{span}}

\title{A refinement of the Cameron-Erd\H{o}s Conjecture}

\date{\today}

\author{Noga Alon}
\address{School of Mathematical Sciences, Tel Aviv University, Tel Aviv 69978, Israel} \email{nogaa@post.tau.ac.il}

\author{J\'ozsef Balogh} 
\address{Department of Mathematics, University of Illinois, 1409 W. Green Street, Urbana, IL 61801} \email{jobal@math.uiuc.edu}

\author{Robert Morris} 
\address{IMPA, Estrada Dona Castorina 110, Jardim Bot\^anico, Rio de Janeiro, RJ, Brasil} \email{rob@impa.br}

\author{Wojciech Samotij}
\address{School of Mathematical Sciences, Tel Aviv University, Tel Aviv 69978, Israel; and Trinity College, Cambridge CB2 1TQ, UK} \email{ws299@cam.ac.uk}

\thanks{Research supported in part by: (NA) an ERC advanced grant, a USA-Israeli BSF grant, and the Israeli I-Core program; (JB) NSF CAREER Grant DMS-0745185, UIUC Campus Research Board Grant 11067, and OTKA Grant K76099; (RM) a CNPq bolsa de Produtividade em Pesquisa; (WS) ERC Advanced Grant DMMCA, and a Trinity College JRF}

\begin{document}

\begin{abstract}
In this paper we study sum-free subsets of the set $\{1,\ldots,n\}$, that is, subsets of the first $n$ positive integers which contain no solution to the equation $x + y = z$. Cameron and Erd\H{o}s conjectured in 1990 that the number of such sets is $O( 2^{n/2} )$. This conjecture was confirmed by Green and, independently, by Sapozhenko. Here we prove a refined version of their theorem, by showing that the number of sum-free subsets of $[n]$ of size $m$ is $2^{O(n/m)} {\lceil n/2 \rceil \choose m}$, for every $1 \le m \le \lceil n/2 \rceil$. For $m \ge \sqrt{n}$, this result is sharp up to the constant implicit in the $O(\cdot)$. Our proof uses a general bound on the number of independent sets of size $m$ in 3-uniform hypergraphs, proved recently by the authors, and new bounds on the number of integer partitions with small sumset. 
\end{abstract}

\maketitle

\section{Introduction}

What is the structure of a typical set of integers, of a given density, which avoids a certain arithmetic sub-structure? This fundamental question underlies much of Additive Combinatorics, and has been most extensively studied when the forbidden structure is a $k$-term arithmetic progression, see e.g.~\cite{Gow1,Gow2,GT1,Roth,SzAP}. General systems of linear equations have also been studied, beginning with Rado~\cite{Rado33} in 1933, and culminating in the recent advances of Green, Tao and Ziegler~\cite{GT2,GTZ}. The subject is extremely rich, and questions of this type have been attacked with tools from a wide variety of areas of mathematics, from Graph Theory to Number Theory, and from Ergodic Theory to Harmonic Analysis. See~\cite{TV} for an excellent introduction to the area.

In this paper we shall consider \emph{sum-free} sets of integers, that is, sets of integers which contain no solution of the equation $x + y = z$. It is easy to see that the odd numbers and the set $\{\lfloor n/2 \rfloor +1,\ldots,n\}$ are the largest such subsets of $[n] = \{1,\ldots,n\}$. Both of these sets have $\lceil n/2 \rceil$ elements, and therefore there are at least $2^{\lceil n/2 \rceil}$ sum-free sets in $[n]$. In 1990, Cameron and Erd\H{o}s~\cite{CE90} conjectured that this trivial lower bound is within a constant factor of the truth, that is, that the set $[n]$ contains only $O(2^{n/2})$ sum-free sets. Despite various attempts~\cite{Noga,Calk,Frei92}, their conjecture remained open for over ten years, until it was confirmed by Green~\cite{G04} and, independently, by Sapozhenko~\cite{Sap03}. We shall prove a natural generalization of the Cameron-Erd\H{o}s Conjecture, by bounding the number of sum-free subsets of $[n]$ of size $m$, for all $1 \le m \le \lceil n/2 \rceil$. Moreover, we shall also give a quite precise structural description of almost all sum-free subsets of $[n]$ of size $m \ge C \sqrt{n \log n}$. Our proof uses a general bound on the number of independent sets of size $m$ in 3-uniform hypergraphs, proved in~\cite{ABMS}, which allows one to deduce asymptotic structural results in the sparse setting (in fact, for all $m \gg \sqrt{n}$) from stability results in the dense setting (see Theorem~\ref{algprop}). The dense stability result we shall use (see Proposition~\ref{prop:Green}) was proved by Green~\cite{G04}. The second main ingredient in the proofs of our main theorems will be some new bounds on the number of sets of integers with small sumset (see Theorems~\ref{S+S} and~\ref{S+S2}). Finally, we shall use Freiman's $3k - 4$ Theorem (see below) to count sets with an extremely small sumset.

The study of sum-free sets of integers dates back to 1916, when Schur~\cite{Schur} proved that if $n$ is sufficiently large, then every $r$-colouring of $[n]$ contains a monochromatic triple $(x,y,z)$ with $x + y = z$. (Such triples are thus often referred to as \emph{Schur triples}.) Sum-free subsets of general Abelian groups have also been studied for many years, see e.g.~\cite{AW,AK,Cam87,Yap69}. Diananda and Yap~\cite{DY} and Green and Ruzsa~\cite{GR05} determined the maximum density $\mu(G)$ of a sum-free set in any finite Abelian group $G$, and in~\cite{GR05} it was moreover shown that any such group $G$ has $2^{(1 + o(1))\mu(G)|G|}$ sum-free subsets. For the group $\ZZ_p$, Sapozhenko~\cite{Sap09} determined the number of sum-free subsets up to a constant factor, and for finite Abelian groups of Type I (those for which $|G|$ has a prime divisor $q \equiv 2 \pmod 3$) Green and Ruzsa~\cite{GR05} were able to determine the asymptotic number of sum-free subsets of $G$.

One of the most significant recent developments in Combinatorics has been the formulation and proof of various `sparse analogues' of classical extremal, structural and Ramsey-type results. Beginning over 20 years ago (see, e.g.,~\cite{BSS,KLR,RR1,RR2}), and culminating in the recent breakthroughs of Conlon and Gowers~\cite{CG} and Schacht~\cite{Sch}, enormous progress has been made in understanding extremal structures in sparse random objects. For example, it is now known (see~\cite{CG,Sch}) that the theorem of Szemer{\'e}di~\cite{SzAP} on $k$-term arithmetic progressions extends to sparse random sets of density $p \gg n^{-1/(k-1)}$, but not to those of density $p \ll n^{-1/(k-1)}$. A sparse analogue of Schur's Theorem was proved by Graham, R\"odl and Ruci\'nski~\cite{GRR}, who showed that if $p \gg 1/\sqrt{n}$ and $B$ is a $p$-random subset\footnote{A $p$-random subset of a set $X$ is a random subset of $X$, where each element is included with probability $p$, independently of all other elements.} of $\ZZ_n$, then with high probability every $2$-colouring of $B$ contains a monochromatic solution of $x + y = z$. A sharp version of this theorem was proved by Friedgut, R\"odl, Ruci\'nski and Tetali~\cite{FRRT}, but the extremal version was open for 15 years before being resolved by Conlon and Gowers~\cite{CG} and Schacht~\cite{Sch}. Even more recently, Balogh, Morris and Samotij~\cite{BMS} sharpened this result by proving that, for any finite Abelian group $G$ of Type I($q$) (that is, $q \equiv 2 \pmod 3$ is the smallest such prime divisor of $|G| = n$) and $pn \ge C(q) \sqrt{n \log n}$, then with high probability every maximum-size sum-free subset of a $p$-random subset of $G$ is contained in some sum-free subset of $G$ of maximum size. In the case $G = \ZZ_{2n}$, they determined the sharp threshold.

For structural and enumerative results, such as our main theorem, results are known in only a few special cases. For example, Osthus, Pr\"omel and Taraz~\cite{OPT} proved that if $m \ge \big( \frac{\sqrt{3}}{4} + \eps \big) n^{3/2} \sqrt{\log n}$, then almost all triangle-free graphs with $m$ edges are bipartite, and that the constant $\sqrt{3}/4$ is best possible. This result can be seen as a sparse version of the classical theorem of Erd\H{o}s, Kleitman and Rothschild~\cite{EKR}, which states that almost all triangle-free graphs are bipartite.  In~\cite{ABMS}, the authors proved a sparse analogue of the result of Green and Ruzsa~\cite{GR05} mentioned above, by showing that if $m \ge C(q) \sqrt{n \log n}$, then almost every sum-free $m$-subset\footnote{An $m$-subset of a set $X$ is simply a subset of $X$ of size $m$.}  of $G$ is contained in some maximum-size sum-free set. We remark that there are only at most $|G|$ maximum-size sum-free subsets of such a group $G$, and that moreover they admit an elegant description.

In this paper we shall be interested in the corresponding question for the set $[n]$. As noted above, Cameron and Erd\H{o}s~\cite{CE90} conjectured, and Green~\cite{G04} and Sapozhenko~\cite{Sap03} proved, that there are only $O(2^{n/2})$ sum-free subsets of $[n]$. Our main result is the following `sparse analogue' of this theorem. 

\begin{thm}\label{CEthm}
There exists a constant $C > 0$ such that, for every $n \in \N$ and every $1 \le m \le \lceil n/2 \rceil$, the set $[n]$ contains at most $2^{Cn/m} {\lceil n/2 \rceil \choose m}$ sum-free sets of size $m$.
\end{thm}

If $m \ge \sqrt{n}$, then Theorem~\ref{CEthm} is sharp up to the value of $C$, since in this case there is a constant $c > 0$ such that there are at least $2^{c n/m} {n/2 \choose m}$ sum-free $m$-subsets of $[n]$ (see Proposition~\ref{CElower}). Note that if $m \le \sqrt{n}$ then the result is trivial, since in this case our upper bound is greater than ${n \choose m}$. Since there are fewer than $2^{n/3}$ subsets of $[n]$ with at most $n/100$ elements, Theorem~\ref{CEthm} easily implies the Cameron-Erd{\H{o}}s Conjecture. However, Theorem~\ref{CEthm} only implies that there are $O(2^{n/2})$ sum-free subsets of $[n]$, whereas Green~\cite{G04} and Sapozhenko~\cite{Sap03} proved that there are asymptotically $c(n)2^{n/2}$ such sets, where $c(n)$ takes two different constant values according to whether $n$ is even or odd.  Since for us the parity of $n$ will not matter, we shall assume for simplicity throughout the paper that $n$ is even; the proof in the case $n$ is odd is identical. 

We shall also prove the following structural description of a typical sum-free $m$-subset of $[n]$. Let $O_n$ denote the set of odd numbers in $[n]$. 

\begin{thm}\label{CEstruc}
There exists $C > 0$ such that if $n \in \N$ and $m \ge C \sqrt{n \log n}$, then almost every sum-free subset $I \subset [n]$ of size $m$ satisfies either $I \subset O_n$, or 
$$|S(I)| \, \le \, \frac{Cn}{m} + \omega(n) \qquad \text{and} \qquad \ds\sum_{a \in S(I)} \left( \frac{n}{2} - a \right) \, \le \, \frac{Cn^3}{m^3} + \omega(n),$$
where $S(I) = \{x \in I : x \le n/2\}$, and $\omega(n) \to \infty$ arbitrarily slowly as $n \to \infty$. 
\end{thm}

We remark that the upper bounds on $|S(I)|$ and $k(I) := \sum_{a \in S(I)}(n/2 - a)$ in Theorem~\ref{CEstruc} are sharp up to a constant factor (see Section~\ref{CESec}). Indeed, we shall show that if $m = o(n)$, then almost all sum-free $m$-sets $I \subset [n]$ have $|S(I)| = \Omega(n/m)$ and $k(I) = \Omega(n^3/m^3)$.

Our proof of Theorems~\ref{CEthm} and~\ref{CEstruc} has two main components. The first is a bound on the number of independent $m$-sets in 3-uniform hypergraphs (see Theorem~\ref{algprop}), which was proved in~\cite{ABMS}, and used there to determine the asymptotic number of sum-free $m$-subsets of a finite Abelian group $G$ such that $|G|$ has a prime factor $q \equiv 2 \pmod 3$, for every $m \ge C(q) \sqrt{n \log n}$. Using this theorem, together with a stability result from~\cite{G04} (which follows from a result of Lev, \L uczak and Schoen~\cite{LLS}) it will be straightforward to bound the number of sum-free $m$-sets which contain at least $\delta m$ even numbers, and at least $\delta m$ elements less than $n/2$.

The second component involves counting \emph{restricted integer partitions with small sumset}. Recall that $p(k)$ denotes the number of integer partitions of $k$, so, for example, $p(3) = 3$ since $3 = 2+1 = 1+1+1$. In 1918, Hardy and Ramanujan~\cite{HR} obtained an asymptotic formula for $p(k)$, proving that 
$$p(k) \, = \, \frac{1 + o(1)}{4k\sqrt{3}} e^{\pi \sqrt{2k/3}}.$$
We shall study the following type of `restricted' partition. Let $p_\ell^*(k)$ denote the number of integer partitions of $k$ into $\ell$ distinct parts, i.e., the number of sets $S \subset \N$ such that $|S| = \ell$ and $\sum_{a \in S} a = k$. Thus, for example, $p^*_3(8) = 2$, since $8 = 5+2+1 = 4+3+1$. It is straightforward to show that $p_\ell^*(k) \le \big( \frac{e^2k}{\ell^2} \big)^\ell$, see Lemma~\ref{parts}. 

We shall bound the number of such partitions under the following more restrictive condition. Recall that, given sets $A,B \subset \N$, the sumset $A+B$ is defined to be the set $\{ a + b : a \in A, \, b \in B \}$. The following theorem bounds the number of partitions of $k$ into $\ell$ distinct parts, such that the resulting set $S$ has `small' sumset $S+S$. 

\begin{thm}\label{S+S}
For every $c_0 > 0$ and $\delta > 0$, there exists a $C = C(\delta,c_0) > 0$ such that the following holds. If $\ell^3 \ge Ck$ and $c \ge c_0$, then there are at most
$$2^{\delta \ell} \left( \frac{2cek}{3\ell^2} \right)^\ell$$
sets $S \subset \N$ with $|S| = \ell$, $\sum_{a \in S} a = k$ and $|S + S| \le ck / \ell$.
\end{thm}

Sets with small sumset are a central object of interest in Combinatorial Number Theory, and have been extensively studied in recent years (see, e.g.,~\cite{TV}). It is easy to see that if $A,B \subset \ZZ$, then $|A + B| \ge |A| + |B| - 1$, with equality if and only if $A$ and $B$ are arithmetic progressions with the same common difference. The Cauchy-Davenport Theorem, proved by Cauchy~\cite{Cauchy} in 1813 and rediscovered by Davenport~\cite{Dav} in 1935, says that this result extends to the group $\ZZ_p$; more precisely, that 
$$|A + B| \, \ge \, \min \big\{ |A| + |B| - 1, \, p \big\}.$$
Many extensions of these results are now known; for example, the Freiman-Ruzsa Theorem (see~\cite{Frei59,Ruz94}) states that if $A \subset \ZZ$ and $|A+A| \le C|A|$, then $A$ is contained in a $O(1)$-dimensional generalized arithmetic progression of size $O(|A|)$. This result itself has many generalizations, culminating in the very recent theorem of Breuillard, Green and Tao~\cite{BGT}, which is stated in the language of approximate groups. 

Despite the enormous interest in such problems, very little seems to be known about the \emph{number} of different sets with small sumset (see~\cite{G05}, for example). The following classical result, proved by Freiman~\cite{Frei59} in 1959, implies a bound for sets with so-called `doubling constant' less than $3$. 

\begin{3k-4}\label{3k-4}
If $A \subset \ZZ$ satisfies $|A+A| \le 3|A| - 4$, then $A$ is contained in an arithmetic progression of size at most $|A+A| - |A| + 1$. 
\end{3k-4}

Observe that this implies that, for all $\lambda < 3$, there are at most $2^{o(\ell)} {(\lambda - 1)\ell \choose \ell}$ sets $S \subset \ZZ$ such that $|S| = \ell$ and $|S+S| \le \lambda \ell$, up to equivalence under translation and dilation. (That is, if we assume that $\min(S) = 0$ and $S$ has no common divisor greater than one.) Our final theorem, which also follows from the proof of Theorem~\ref{S+S}, provides a similar bound whenever $|S+S| = O(|S|)$.\footnote{Throughout the paper, $f(n) = O(g(n))$ means there exists an absolute constant, independent of all other variables, such that $f(n) \le Cg(n)$.} The following result will be crucial in the proof of Theorems~\ref{CEthm} and~\ref{CEstruc} in the case $m = \Theta(n)$. 

\begin{thm}\label{S+S2}
Let $\delta > 0$, and suppose that $\ell \in \N$ is sufficiently large and that $k \le \ell^2 / \delta$. Then for each $\lambda \ge 2$, there are at most
$$2^{\delta \ell} \left( \frac{(4\lambda - 3)e}{6} \right)^\ell$$
sets $S \subset \N$ with $|S| = \ell$, $\sum_{a \in S} a = k$, and $|S + S| \le \lambda \ell$.
\end{thm}
 
Theorems~\ref{S+S} and~\ref{S+S2} are sufficient for our purposes; however, we believe the following stronger bound to be true. 

\begin{conj}\label{S+Sconj}
For every $\delta > 0$, there exists $C > 0$ such that the following holds. If $m \ge C\sqrt{N}$ and $m \ge C \log n$, then there are at most
$$2^{\delta m} {N/2 \choose m}$$
sets $S \subset [n]$ with $|S| = m$ and $|S + S| \le N$. 
\end{conj}

Since $|S+S| \le N$ for every $m$-subset $S \subset [N/2]$, the conjecture (if true) is close to optimal. Note that the condition $m \ge C \log n$ implies that $n \le 2^{m/C}$, and thus guarantees that the number of translates of a given set $S$ is negligible. 

The rest of the paper is organised as follows. In Section~\ref{GenThmSec}, we shall recall the general structural theorem from~\cite{ABMS} and deduce from it a bound on the number of sum-free $m$-sets which contain at least $\delta m$ even numbers, and at least $\delta m$ elements less than $n/2$. In Section~\ref{lowerSec} we shall prove a lower bound on the number of sum-free $m$-subsets of $[n]$, and in Section~\ref{JansonSec} we shall use Janson's inequality to bound the number of sum-free sets which contain at most $\delta m$ even numbers. In Section~\ref{partSec} we shall prove Theorems~\ref{S+S} and~\ref{S+S2}. Finally, in Section~\ref{CESec}, we shall prove Theorems~\ref{CEthm} and~\ref{CEstruc}.

\section{Preliminaries}\label{GenThmSec}

In this section we shall recall some of the main tools we shall use in the proofs of Theorems~\ref{CEthm} and~\ref{CEstruc}, and deduce that almost all sum-free $m$-sets $I \subset [n]$  either contain at most $\delta m$ even elements, or satisfy $| I \setminus B | \le \delta m$ for some interval $B$ of length $n/2$.

\subsection{A structural theorem for 3-uniform hypergraphs}

We begin by recalling from~\cite{ABMS} our main tool: a theorem which allows one to deduce asymptotic structural results for sparse sum-free sets from stability results for dense sum-free sets. It is stated in the language of general $3$-uniform (sequences of) hypergraphs $\HH = (\HH_n)_{n \in \N}$, where $|V(\HH_n)| = n$. Throughout this section, the reader should think of $\HH_n$ as encoding the Schur triples (that is, triples $(x,y,z)$ with $x + y = z$) in $[n]$.

We next recall the dense stability property which we shall use. Let $\alpha \in (0,1)$ and let $\B = (\B_n)_{n \in \N}$, where $\B_n$ is a family of subsets of $V(\HH_n)$. We shall write $|\B_n|$ for the number of sets in $\B_n$, and set $\|\B_n\| = \max\{ |B| : B \in \B_n\}$. 

\begin{defn}
A sequence of hypergraphs $\HH = (\HH_n)_{n \in \N}$ is said to be \emph{$(\alpha,\B)$-stable} if for every $\gamma > 0$ there exists $\beta > 0$ such that the following holds. If $A \subset V(\HH_n)$ satisfies $|A| \ge (\alpha - \beta) n$, then either $e(\HH_n[A]) \ge \beta e(\HH_n)$, or $|A \setminus B| \le \gamma n$ for some $B \in \B_n$.
\end{defn}

Roughly speaking, a sequence of hypergraphs $(\HH_n)$ is $(\alpha, \B)$-stable if for every $A \subseteq V(\HH_n)$ such that $|A|$ is almost as large as the independence number for $\HH_n$, the set $A$ is either very close to some `extremal' set $B \in \B_n$, or it contains many (i.e., a positive fraction of all) edges of $\HH_n$. 

If $\HH_n$ is a hypergraph and $m \in \N$, then let $\SF(\HH_n,m)$ denote the collection of independent sets in $\HH_n$ of size $m$. Given a family of sets $\B_n$ and $\delta > 0$, we define
$$\SF^{(\delta)}_{\ge}(\HH_n,\B_n,m) \, = \, \Big\{ I \in \SF(\HH_n,m) \,\colon\, |I \setminus B| \ge \delta m \textup{ for every } B \in \B_n \Big\}.$$
Finally, for each $T \subset V(\HH_n)$, let $d_{\HH_n}(T) = \big| \big\{ e \in \HH_n \,:\, T \subset e \big\} \big|$ and define
$$\Delta_2(\HH_n) \,=\, \max \big\{ d_{\HH_n}(T) \,:\, T \subset V(\HH_n), \, |T| = 2 \big\}.$$
Note that if $\HH_n$ encodes Schur triples in $[n]$, then $\Delta_2(\HH_n) \le 2$. 

The following theorem, which was proved in~\cite{ABMS}, shows that if $\HH$ is $(\alpha,\B)$-stable and $m \gg \sqrt{n}$, then there are very few independent sets (i.e., sum-free sets) in $\HH_n$ of size $m$ which are far from every set $B \in \B_n$.

\begin{thm}[Theorem~4.1 of~\cite{ABMS}]\label{algprop}
Let $\alpha > 0$ and let $\HH = (\HH_n)_{n \in \N}$ be a sequence of $3$-uniform hypergraphs which is $(\alpha,\B)$-stable, has $e(\HH_n) = \Theta(n^2)$ and $\Delta_2(\HH_n) = O(1)$. If $\|\B_n\| \ge \alpha n$, then for every $\delta > 0$, there exists a $C > 0$ such that the following holds. If $m \ge C\sqrt{n}$ and $n$ is sufficiently large, then 
$$\big| \SF^{(\delta)}_{\ge}(\HH_n,\B_n,m) \big| \,\le \, \Big( 2^{-\eps m} + \delta^m |\B_n| \Big) {\|\B_n\| \choose m}$$
for some $\eps = \eps(\HH, \delta) > 0$.
\end{thm}

In the next subsection, we shall use this theorem, together with a result of Green~\cite{G04}, to deduce an approximate version of Theorem~\ref{CEstruc}. 

\subsection{Green's stability theorem}

Let $\HH = (\HH_n)_{n \in \N}$ be the sequence of hypergraphs which encodes Schur triples in $[n]$; that is, $V(\HH_n) = [n]$ and $\{x,y,z\} \in E(\HH_n)$ whenever $x + y = z$. The following stability result, due to Green~\cite{G04}, implies that $\HH$ is $(\alpha,\B)$-stable, where $\alpha = 1/2$ and
\begin{equation}
  \label{eq:CE-Bn-def}
  \B_n \, = \, \Big\{ \big\{ a+1,\ldots,a+ n/2 \big\} \,\colon\, 0 \le a \le n/2 \Big\} \cup \big\{ O_n \big\},
\end{equation}
where, as before, $O_n$ denotes the odd numbers in $[n]$.

\begin{prop}[{\cite[Proposition~7]{G04}}]\label{prop:Green}
For any $\gamma > 0$, if $\beta = \beta(\gamma) > 0$ is sufficiently small, then the following holds. If $A \subset [n]$ with $|A| \ge (1/2 - \beta)n$, then either $A$ contains at least $\beta n^2$ Schur triples, or $|A \setminus B| \le \gamma n$ for some $B \in \B_n$. 
\end{prop}

Using Theorem~\ref{algprop} and Proposition~\ref{prop:Green}, we easily obtain the following corollary.

\begin{prop}\label{CEprop}
For every $\delta > 0$, there exist constants $C > 0$ and $\eps > 0$ such that the following holds for every $n \in \N$ and $m \ge C\sqrt{n}$. There are at most 
$$2^{-\eps m} {n/2 \choose m}$$ 
sum-free subsets $I \subset [n]$ of size $m$ such that $| I \setminus B | > \delta m$ for every $B \in \B_n$.

In particular, for almost every sum-free set $I \subset [n]$ of size $m$, either $| I \setminus O_n | \le \delta m$, or $| I \setminus B | \le \delta m$ for some interval $B$ of length $n/2$.
\end{prop}

\begin{proof}
Let $\HH = (\HH_n)_{n \in \N}$ be the sequence of hypergraphs which encodes Schur triples in $[n]$, as above, and let $\B = (\B_n)_{n \in \N}$ be the collection of intervals of length $n/2$, plus the odds, as in~\eqref{eq:CE-Bn-def}. Set $\alpha = 1/2$, and observe that $\HH$ is $(\alpha,\B)$-stable, by Proposition~\ref{prop:Green}. 

Now, by Theorem~\ref{algprop}, if $C = C(\delta) > 0$ is sufficiently large, and $m \ge C\sqrt{n}$, then 
$$\big| \SF^{(\delta)}_{\ge}(\HH_n,\B_n,m) \big| \,\le \, 2^{-\eps m} {\|\B_n\| \choose m} \, = \, 2^{-\eps m} {n/2 \choose m}$$
for some $\eps = \eps(\delta) > 0$. Since there are at least ${n/2 \choose m}$ sum-free $m$-subsets of $[n]$, it follows that for almost every such set $I$ we have $|I \setminus B| \le \delta n$ for some $B \in \B_n$, as required.
\end{proof}

We remark that it will be relatively straightforward to count the sets $I$ that contain fewer than $\delta m$ even elements, using Janson's inequality (see Section~\ref{JansonSec}), and those that contain more than $\delta m$ elements less than $n/2$, using induction on $n$ (see Section~\ref{CESec}). Thus, Proposition~\ref{CEprop} essentially reduces the problem of counting sum-free $m$-sets in $[n]$ to counting the sum-free sets that are almost contained in the interval $\{n/2+1,\ldots,n\}$.

\subsection{Binomial coefficient inequalities}

We shall make frequent use of some simple inequalities involving binomial coefficients; for convenience, we collect them here. Note first that ${a \choose b} \le \big( \frac{ea}{b} \big)^b$ and that ${a \choose b}$ is increasing in $a$. Next, observe that if $a > b > c \ge 0$, then
\begin{equation}\label{trivial}
{a \choose b - c} \, \le \, \left( \frac{b}{a-b} \right)^c {a \choose b} \qquad \text{and} \qquad {a - c \choose b} \, \le \, \left( \frac{a - c}{a} \right)^b {a \choose b},
\end{equation}
and hence
\begin{equation}\label{trivial2}
{a - c \choose b - d} \, \le \, \left( \frac{a - c}{a} \right)^{b-d} \left( \frac{b}{a-b} \right)^d {a \choose b}.
\end{equation}
We shall also use several times the observation that
\begin{equation}\label{eq:kmax}
\sum_{k=1}^{\infty} k^a e^{-bk} \, \le \, c \cdot \frac{\Gamma(a+1)}{b^{a+1}} \, \le \, C \left( \frac{a}{b} \right)^{a+1} e^{-(a+1)}, 
\end{equation}
where $\Gamma(\cdot)$ is Euler's Gamma function, for some $C > c > 0$ and every $a \ge 1$ and $b > 0$.

For other standard probabilistic bounds, such as the FKG inequality and Chernoff's inequality, we refer the reader to~\cite{AS}.

\section{A lower bound on the number of sum-free sets}\label{lowerSec}

In this section we shall prove the following simple proposition, which shows that the bound in Theorem~\ref{CEthm} is tight. 

\begin{prop}\label{CElower}
If $m \ge \sqrt{n}$, then there are $2^{\Omega(n/m)} {n/2 \choose m}$ sum-free subsets of $[n]$ of size $m$.
\end{prop}

\begin{proof}
Let $c > 0$ be a sufficiently small absolute constant and set $a = c n^2 / m^2$. We claim that if $S$ is a uniformly chosen random $m$-subset of $U = \{n/2 - a, \ldots, n\}$, then 
\begin{equation}\label{FKGeq}
\Pr\big( S\text{ is sum-free} \big) \, \ge \, \exp\left( - \frac{cn}{2m} \right).
\end{equation}
In order to prove~\eqref{FKGeq}, we shall in fact choose the elements of $S$ independently at random with probability $p = 4m/n$, and bound $\Pr_p\big( S\text{ is sum-free} \,|\, |S| = m \big)$, which is clearly equivalent. (Note that the proposition is trivial if $m = \Omega(n)$, so we may assume that $p$ is sufficiently small.)

First observe that there are at most $a^2+a$ triples $\{x,y,z\}$ in $U$ with $x + y = z$, and at most $a+1$ pairs $\{x,y\}$ in $U$ with $2x = y$. Thus, by the FKG inequality, 
$$\Pr_p\big( S \text{ is sum-free} \big) \,\ge\, \big(1 - p^3 \big)^{a^2+a}\big(1 - p^2 \big)^{a+1} \,\ge\, \exp\left( - \frac{cn}{3m} \right)$$ 
since $c > 0$ is sufficiently small, by our choices of $a$ and $p$. Next, note that, by Chernoff's inequality,
$$\Pr_p\big( |S| < m \big) \, \le \, e^{-cm} \,\le\, e^{-cn/m},$$ 
since $m \ge \sqrt{n}$. Finally, observe that $g(t) = \Pr_p\big( S\text{ is sum-free} \,|\, |S| = t \big)$ is decreasing in $t$.

It follows immediately that 
$$\Pr_p\big( S\text{ is sum-free} \,|\, |S| = m \big) \, \ge \, \exp \left( - \frac{cn}{2m} \right),$$
which proves~\eqref{FKGeq}. Hence the number of sum-free $m$-sets in $\{n/2 - a, \ldots, n\}$ is at least
$${ n/2 + a \choose m }\exp \bigg( - \frac{cn}{2m} \bigg) \,\ge\, {n/2 \choose m} \exp\bigg( \frac{am}{n} - \frac{cn}{2m} \bigg) \,=\, \exp \bigg( \frac{cn}{2m} \bigg) {n/2 \choose m},$$
where the inequality follows from~\eqref{trivial} and the fact that $\big(1+\frac{2a}{n}\big) \ge e^{a/n}$.
\end{proof}

\section{Janson argument}\label{JansonSec}

In this section, we shall count the sum-free sets that have few even elements. Recall that $O_n$ denotes the odd numbers in $[n]$. 

\begin{prop}\label{prop:Janson2}
If $\delta > 0$ is sufficiently small, then there are at most $2^{O(n/m)} {n/2 \choose m}$ sum-free subsets $I \subset [n]$ with $|I| = m$ and $|I \setminus O_n| \le \delta m$, for every $m,n \in \N$.
\end{prop}

We remark that an argument similar to the one presented in this section was used in~\cite{ABMS} in a somewhat more general context, see also~\cite{BMS}. Indeed, the following result was proved in~\cite{ABMS}.

\begin{prop}[{\cite[Proposition~5.1]{ABMS}}]\label{prop:Janson}
There exists constants $\delta > 0$ and $C > 0$ such that the following holds for every $m \ge C\sqrt{n \log n}$. There are at most 
$$\left( 1 + \frac{1}{n^3} \right) {n/2 \choose m}$$
sum-free subsets $I \subset [n]$ with $|I| = m$ and $|I \setminus O_n| \le \delta m$.
\end{prop}

Proposition~\ref{prop:Janson} clearly implies Proposition~\ref{prop:Janson2} in the case $m \ge C\sqrt{n \log n}$. Furthermore, the proposition is trivial if $m \le O\big( \sqrt{n} \big)$, since then the claimed upper bound is greater than $\binom{n}{m}$. Thus, we need only consider the case $C\sqrt{n} \le m \le C\sqrt{n \log n}$. 

Recall the following well-known result, which is an easy corollary of Janson's inequality (see~\cite{AS,JLR}), combined with Pittel's inequality (see~\cite{JLR}). We refer the reader to~\cite[Section~5]{ABMS} for a proof.

\begin{lemma}[Hypergeometric Janson Inequality]\label{HJI}
Suppose that $\{U_i\}_{i \in J}$ is a family of subsets of an $n$-element set $X$ and let $m \in \{0, \ldots, n\}$. Let
$$\mu = \sum_{i \in J} (m/n)^{|U_i|} \quad \text{and} \quad \Delta = \sum_{i \sim j} (m/n)^{|U_i \cup U_j|},$$
where the second sum is over ordered pairs $(i,j)$ such that $i \neq j$ and $U_i \cap U_j \neq \emptyset$. Let $R$ be a uniformly chosen random $m$-subset of $X$. Then
$$\Pr\big( U_i \nsubseteq R \text{ for all $i \in J$} \big) \,\le\, C \cdot \max\left\{e^{-\mu/2}, e^{-\mu^2/(2\Delta)}\right\},$$
for some absolute constant $C > 0$.
\end{lemma}

We now turn to the proof of Proposition~\ref{prop:Janson2}.\medskip

Let $C > 0$ be a sufficiently large constant, and recall that we may assume that $C \sqrt{n} \le m \le C\sqrt{n \log n}$. We begin by proving the following claim.

\begin{claim*}
For some constant $c > 0$, there are at most
\begin{equation}\label{eq:PJ2}
C \cdot {n/2 \choose k} \max\Big\{ e^{-c k m^2  / n}, e^{-c m} \Big\} {n/2 \choose {m-k}}
\end{equation}
sum-free $m$-sets $I \subset [n]$ with $|I \setminus O_n| = k \le \delta m$. 
\end{claim*}

\begin{proof}[Proof of claim]
Let $k \le \delta m$ and let $S$ be an arbitrary $k$-subset of $[n] \setminus O_n$. Let $\{U_i\}_{i \in J}$ be the collection of pairs $\{x,y\} \subset O_n$ such that either $x + y = z$ or $x - y = z$ for some $z \in S$. In order to bound the number of sum-free $m$-sets $I$ with $I \setminus O_n = S$, we shall apply the Hypergeometric Janson Inequality to the collection $\{U_i\}_{i \in J}$ and the set $X = O_n$, with $R$ a uniformly chosen random $(m-k)$-subset of $X$. Note that if $S \cup R$ is sum-free, then $U_i \nsubseteq R$ for all $i \in J$. 

Let $\mu$ and $\Delta$ be the quantities defined in the statement of Lemma~\ref{HJI}, and observe that for every even number $z$, there are either at least $n/10$ pairs $\{x,y\} \subset O_n$ with $x + y = z$ (if $z \ge n/2$), or at least $n/5$ such pairs with $x - y = z$ (if $z \le n/2$). Thus $nk/20 \le |J| \le nk$, since each pair can be counted at most twice. Observe that each vertex $x \in O_n$ lies in at most $2k$ of the $U_i$. Hence
$$\mu \, \ge \, \frac{nk}{20} \cdot \frac{(m-k)^2}{n^2} \, \ge \, \frac{k m^2}{30n} \quad \text{and} \quad \Delta \,\le\, (2k)^2 \left( \frac{|J|}{2k} \right) \bigg( \frac{m}{n} \bigg)^3 \le \, \frac{2k^2m^3}{n^2}.$$
By the Hypergeometric Janson Inequality, if $c = 10^{-4}$ then there are at most
$$C \cdot \max\Big\{ e^{-c k m^2  / n}, e^{-c m} \Big\}  {n/2 \choose m-k}$$
sets $R \subset O_n$ of size $m - k$ such that $S \cup R$ is sum-free. Summing over choices of $S$, we obtain the claimed bound.
\end{proof}

Now, by~\eqref{trivial} and since $m \le C\sqrt{n \log n} \le n/6$, if $k \ge n/m$ then~\eqref{eq:PJ2} is at most
$$C \cdot {n/2 \choose k} \left(\frac{m}{n/2-m} \right)^k e^{-cm} {n/2 \choose m} \, \le \, C \cdot \left( \frac{3em}{2k} \right)^k e^{-cm} {n/2 \choose m} \, \ll \,  {n/2 \choose m},$$
assuming $\delta > 0$ is sufficiently small. However, if $k \le n/m$ then~\eqref{eq:PJ2} is at most
$$C \cdot \left( \frac{3em}{2k} e^{-cm^2/n} \right)^k {n/2 \choose m} \, \le \,  2^{O(n/m)} {n/2 \choose m}.$$
To see the final inequality, observe that (since $x e^{-x/n} \le n$) we have $m e^{-c m^2/n} \le n/cm$, and use the fact that $k \mapsto (a/k)^k$ is maximized when $k = a/e$. This completes the proof of Proposition~\ref{prop:Janson2}.

\section{Partitions and sumsets}\label{partSec}

In this section we shall prove Theorems~\ref{S+S} and~\ref{S+S2}. Recall that
$$p_\ell^*(k) \, = \, \#\big\{ \text{partitions of $k$ into $\ell$ distinct parts} \big\}.$$ 
We shall use the following easy upper bound on $p^*_\ell(k)$ in the proofs of Theorems~\ref{CEthm} and~\ref{CEstruc}. 

\begin{lemma}\label{parts}
For every $k, \ell \in \N$,
$$p_\ell^*(k) \, \le \, \left( \frac{e^2k}{\ell^2} \right)^\ell.$$
\end{lemma}

\begin{proof}
Consider putting $k$ identical balls into $\ell$ labelled boxes. There are ${k + \ell - 1 \choose \ell - 1} = \frac{\ell}{k+\ell} {k + \ell \choose \ell}$ ways to do so, and each partition of $k$ into $\ell$ distinct parts is counted exactly $\ell!$ times. Using the bound $\ell! \ge \sqrt{2\pi \ell} \big( \frac{\ell}{e} \big)^\ell$, it follows that 
 $$p_\ell^*(k) \, \le \, \frac{1}{\ell!} \cdot \frac{\ell}{k + \ell} {k + \ell \choose \ell}  \, \le \,  \frac{2}{\ell+2} \cdot \frac{1}{\sqrt{2\pi\ell}} \bigg( \frac{e}{\ell} \bigg)^\ell \left( \frac{e(k + \ell)}{\ell} \right)^{\ell}  \, \le \, \left( \frac{e^2k}{\ell^2} \right)^\ell,$$
 if $k \ge {\ell + 1 \choose 2}$ and $\ell \ge 4$, since $\big( \frac{k+\ell}{k} \big)^{\ell} < e^2 < 6\sqrt{2\pi}$. Finally, note that the result is trivial if $\ell \le 3$, and that $p_\ell^*(k) = 0$ if $k < {\ell + 1 \choose 2}$. 
\end{proof}

In order to motivate the proofs of Theorems~\ref{S+S} and~\ref{S+S2}, we shall first sketch an easy proof of a weaker bound and an incorrect proof of a sharper one; we will use ideas from both in the actual proof. We begin with the weaker bound: given $c,\delta > 0$, let $C = C(\delta,c) > 0$ be sufficiently large, and suppose that $\ell^3 \ge Ck$. We claim that there are at most
\begin{equation}\label{S+Sweak}
2^{\delta \ell} \left( \frac{cek}{\ell^2} \right)^\ell
\end{equation}
sets $S \subset \N$ with $|S| = \ell$, $\sum_{a \in S} a = k$, and $|S + S| \le ck / \ell$. Note that this is weaker than Theorem~\ref{S+S} by a factor of $(3/2)^\ell$. 

We shall count `good' sequences $(a_1,\ldots,a_\ell)$ of length $\ell$, that is, sequences such that the underlying set $S = \{a_1,\ldots,a_\ell\}$ satisfies $|S| = \ell$, $\sum_{a \in S} a = k$ and $|S + S| \le ck / \ell$. Note that each such set $S$ will appear as a sequence exactly $\ell!$ times. Set $S_j = \{a_1,\ldots,a_j\}$ and observe that $\left| \big( S_j + a_{j+1} \big) \setminus \big(S_j + S_j \big) \right| \ge \delta |S_j|$ for at most $\delta \ell$ indices $j \in [\ell]$, since otherwise $|S+S| \ge \delta (\delta \ell / 2)^2 > ck / \ell$, where the last inequality follows because $\ell^3 \ge Ck$.

We now make a simple but key observation: that, for every set $S \subset \N$, there are at most $(1 - \delta)^{-1}|S+S|$ elements $y \in \N$ such that 
\begin{equation}\label{eq:S+y}
\left| \big( S + y \big) \setminus \big( S + S \big) \right| \, \le \, \delta |S|.
\end{equation}
To prove this, observe that there are $|S| \cdot |S+S|$ pairs $(a,b)$ with $a \in S$ and $b \in S+S$, and that if~\eqref{eq:S+y} holds then $a + y = b$ for at least $(1 - \delta)|S|$ pairs $(a,b) \in S \times (S+S)$. For each pair $(a,b)$ there is at most one such $y$, and so there are at most $\frac{|S| \cdot |S+S|}{(1 - \delta)|S|}$ elements $y$ which satisfy~\eqref{eq:S+y}, as claimed. 

Thus, as in the proof of Lemma~\ref{parts}, the number of good sequences $(a_1,\ldots,a_\ell)$ is at most 
\begin{eqnarray}\label{eq:sketch}
{\ell \choose \delta \ell} {k + \delta \ell \choose \delta \ell} \left( \frac{|S+S|}{1 - \delta} \right)^{(1 - \delta)\ell} & \le & \left( \frac{e}{\delta} \right)^{\delta \ell} \left( \frac{2ek}{\delta \ell} \right)^{\delta \ell} \left( \frac{1}{|S+S|} \right)^{\delta \ell}  e^{O(\delta \ell)} |S+S|^{\ell} \nonumber \\
& \le & \left( \frac{2e^2}{\delta^2 c} \right)^{\delta \ell}  e^{O(\delta \ell)}  \left( \frac{ck}{\ell} \right)^{\ell} \, \le \,  2^{O(\sqrt{\delta} \ell)} \left( \frac{ck}{\ell} \right)^{\ell},
\end{eqnarray}
assuming $\delta > 0$ is sufficiently small. Dividing by $\ell!$, we obtain~\eqref{S+Sweak}, as claimed.

Next, define the \emph{span} of a set $S \subset \N$ to be $\max(S) - \min(S)$ and observe that, for any set $S$, we have $2 \cdot \span(S) = \span(S + S)$. Let $B_j$ denote the set of elements $y$ (as above) for which $|(S_j + y) \setminus (S_j+S_j)| \le \delta |S_j|$. Intuitively, one would expect that $B_j + S_j \approx S_j + S_j$, which implies that
\begin{equation}\label{intuitively}
\span(B_j) + \span(S_j) \, = \, \span(B_j + S_j) \, \approx \, \span(S_j + S_j) \, = \, 2 \cdot \span(S_j),
\end{equation}
and hence $\span(B_j) \approx \span(S_j)$. This would imply that $\min(S_j) + B_j$ and $\max(S_j) + B_j$ are (almost) disjoint, since 
$$\max\big( \min(S_j) + B_j \big) \, = \, \min(S_j) + \max(B_j) \, \approx \, \max(S_j) + \min(B_j) \, = \, \min\big( \max(S_j) + B_j \big).$$ 
But now, if $B_j + S_j \approx S_j + S_j$, then
$$|B_j| \, \lesssim \, \frac{|S_j + S_j|}{2},$$ 
which would win us a factor of roughly $2^\ell$ in~\eqref{eq:sketch}. Unfortunately,~\eqref{intuitively} does not hold in general; for example, if $S_j = [k] \cup X$, where $X$ is a random subset of $\{k+1,\ldots,2k\}$ of size $\delta k$, then $\span(S_j) \approx 2k$ but $\span(B_j) \approx 3k$. Nevertheless, we shall be able to prove an approximate version of~\eqref{intuitively} (see~\eqref{spanBJ}, below) by considering a subset $J \subset S_j$ which is sufficiently dense close to its extremal values, $\max(J)$ and $\min(J)$.

\begin{proof}[Proof of Theorem~\ref{S+S}]
Let $c \ge c_0 > 0$ and $\delta > 0$, and note that without loss of generality we may assume that $\delta = \delta(c_0)$ is sufficiently small. Let $C = C(\delta,c_0) > 0$ be sufficiently large; with foresight, we remark that $C = 1/\delta^{13}$ will suffice. Note also that if $c \ge 3e/2$, then the theorem follows immediately from Lemma~\ref{parts}; we shall therefore assume that $c < 3e/2$.

Given $S \subset \N$ with $|S| = \ell$, $\sum_{a \in S} a = k$ and $|S+S| \le ck/\ell$, let $S_* = \max\big\{ \delta k / \ell, \,\min(S) \big\}$ and $S^* = \min\big\{ k/ \delta \ell, \,\max(S) \big\}$, so $\delta k / \ell \le S_* \le S^* \le k / \delta \ell$. Moreover, define
$$t_* \, = \, \min\Big\{ t \,\colon \big| S \cap \big[ (1 + \delta)^t S_*,(1 + \delta)^{t+1} S_* \big] \big| \ge \delta^3 \ell \Big\}$$
and $t^* = \min\Big\{ t \,\colon \big| S \cap \big[ (1 - \delta)^{t+1} S^*,(1 - \delta)^{t} S^* \big] \big| \ge \delta^3 \ell \Big\}$ (or $\infty$ if no such $t$ exists), and set 
\begin{equation}\label{Jdef}
J \,=\, [J_*,J^*] \, = \, \big[ (1 + \delta)^{t_*} S_*,(1 - \delta)^{t^*} S^* \big].
\end{equation}
Note that the set $J$ has `dense ends', that is, the sets $[(1- \delta)J^*,J^*]$ and $[J_*,(1 + \delta)J_*]$ each contain more than $\delta^3 \ell$ elements of $S$. We shall apply the argument which failed to work above to the set $S_J = S \cap J$. 

\bigskip
\noindent \textbf{Case 1:} $\max\big\{ t^*,t_* \big\} \ge 1 / \delta^2$.
\medskip

Suppose first that $t^* \ge 1/\delta^2$. Note that $S$ contains at most $\delta \ell$ elements greater than $S^*$, since $\sum_{a \in S} a = k$, and hence at most $2\delta \ell$ elements of $S$ greater than 
$$(1 - \delta)^{1/\delta^2} S^* \, \le \, e^{-1/\delta} \frac{k}{\delta \ell} \, \le \, \frac{\delta k}{\ell}.$$ 
Let $s \le 2\delta \ell$ denote the number of elements of $S$ greater than $\delta k / \ell$, and note that ${a \choose \ell - s} \le \big( \frac{e^2 a}{\ell} \big)^{\ell - s}$, since $e(\ell - s) \ge \ell$. By Lemma~\ref{parts}, it follows that there are at most  
\begin{equation}\label{eq:S+S}
\left( \frac{e^2k}{s^2} \right)^s {\delta k/\ell \choose \ell - s} \, \le \, \left( \frac{e^2k}{s^2} \right)^s \bigg( \frac{e^2 \delta k}{\ell^2} \bigg)^{\ell-s}  \, = \, \left( \frac{\ell^2}{\delta s^2} \right)^s \bigg( \frac{e^2 \delta k}{\ell^2} \bigg)^{\ell} \, \le \, \bigg( \frac{\sqrt{\delta} k}{\ell^2} \bigg)^\ell,
\end{equation}
such sets $S$, where the last inequality follows since $\delta > 0$ is sufficiently small. Summing over~$s$, it follows that there are at most 
$\ell \big( \frac{\sqrt{\delta} k}{\ell^2} \big)^\ell$ sets $S$ with $t^* \ge 1/\delta^2$. Since $\delta = \delta(c_0)$ was chosen sufficiently small, the required bound follows.

Similarly, if $t_* \ge 1 / \delta^2$, then $S$ contains at most $2\delta \ell$ elements greater than $\delta k / \ell$ (at most $\delta \ell$ in the range $[\delta k / \ell, k / \delta \ell] \subset [\delta k / \ell, (1+\delta)^{1/\delta^2} \delta k / \ell]$, and at most $\delta \ell$ greater than $k / \delta \ell$). Thus $\ell \big( \frac{\sqrt{\delta} k}{\ell^2} \big)^\ell$ is also an upper bound on the number of sets $S$ with $t_* \ge 1/\delta^2$. 

\bigskip
\noindent \textbf{Case 2:} $\max\big\{ t^*,t_* \big\} \le 1 / \delta^2$.
\medskip

Since $t^*, t_* \le 1/\delta^2$, there are at most $3\delta \ell$ elements of $S$ outside the set $[0,\delta k / \ell] \cup J$. Moreover, we may assume that $S$ has at least $\ell / 3$ elements larger than $\delta k / \ell$, see~\eqref{eq:S+S}, and hence that $|S_J| \ge \ell/4$. Thus, since $k = \sum_{a \in S} a \ge |S_J| \cdot J_*$, it follows that $J_* \le 4k/ \ell$. 

Set $J_0 = \big\{ a \in S : a \le \delta k / \ell \big\}$, let $b = |J_0|$ and set $r = |S \setminus (J \cup J_0)| \le 3\delta \ell$. Suppose first that $\span(J) < ck/8\ell$. Then, by Lemma~\ref{parts}, the number of choices for $S$ is at most\footnote{Here, as usual, $0^0 = 1$.}
\begin{multline*}\label{C2:eq1}
\sum_{b = 0}^{2\ell/3} \sum_{r = 0}^{3\delta \ell} \sum_{J_*,J^*} {\delta k / \ell \choose b} \left( \frac{e^2 k}{r^2} \right)^r {ck/8\ell \choose \ell - b - r} \, \le \, k^2 \ell^2 \max_{b  \le 2 \ell / 3, \, r \le 3\delta \ell} \left( \frac{e \delta k}{\ell b} \right)^b \left( \frac{e^2 k}{r^2} \right)^r \left( \frac{c e k}{2\ell^2} \right)^{\ell - b - r}\\
= \, k^2 \ell^2 \max_{b  \le 2 \ell / 3, \, r \le 3\delta \ell} \left( \frac{2 \delta \ell}{bc} \right)^b \left( \frac{2e\ell^2}{cr^2} \right)^r  \left( \frac{c e k}{2\ell^2} \right)^\ell \, \le \, 2^{\sqrt{\delta} \ell}  \left( \frac{cek}{2\ell^2} \right)^\ell
\end{multline*}
if $\delta = \delta(c_0) > 0$ is sufficiently small, since the maximum occurs at $r = 3\delta \ell$ and $b = 2\delta \ell / (ec)$. 

Thus we may assume that $\span(J) \ge ck/8\ell$, from which it follows that
\begin{equation}\label{spanJ}
\span(J) \, = \, J^* - J_* \, \ge \, \sqrt{\delta} \big( J^* + J_* \big),
\end{equation} 
since $J_* \le 4k/ \ell$, and so $J^* \ge J_* + ck / 8\ell \ge \big( 1 + c/32 \big) J_*$. 

\medskip
Let us count sequences $\a = (a_1,\ldots,a_\ell)$ of distinct elements such that $S = \{a_1,\ldots,a_\ell\}$ satisfies $\sum_{a \in S} a = k$ and $|S+S| \le ck/\ell$. Given such a sequence $\a$, for each $j \in [\ell]$, set 
$$S_j \, = \, \big\{ a_1,\ldots,a_j \big\} \cap J,$$
where $J$ was defined in~\eqref{Jdef}, and define 
\begin{equation}\label{Bdef}
B_j \, = \; \Big\{ y \in \N \,:\,  \left| \big( S_j + y \big) \setminus \big( S_j + S_j \big) \right| \le \delta^6 |S_j| \Big\}.
\end{equation}
We make the following key claim.

\begin{claim*} 
Suppose that the intervals $[J_*,(1+\delta)J_*]$ and $[(1- \delta)J^*,J^*]$ contain more than $\delta^5 \ell$ elements of $\big\{a_1,\ldots,a_{\delta \ell} \big\}$ each. Then
$$|B_j| \,\le \, \left( \ds\frac{2}{3} + \delta \right) |S+S|$$ 
for every $\delta \ell \le j \le \ell$. 
\end{claim*}

\begin{proof}[Proof of claim]
Fix $j$ with $\delta \ell \le j \le \ell$. Recall from~\eqref{Bdef} that $\left| \big( S_j + y \big) \setminus \big( S_j + S_j \big) \right| \le \delta^6 \ell$ for every $y \in B_j$, and that $2 \cdot \span(S_j) = \span(S_j + S_j)$. We claim that
\begin{equation}\label{spanBJ}
\big( 1 - \sqrt{\delta} \big)\span(J) + \span(B_j) \, \le \, \span(S_j + S_j).
\end{equation}
Indeed, since $[(1- \delta)J^*,J^*]$ contains more than $\delta^5 \ell$ elements of $\big\{ a_1, \ldots, a_{\delta \ell} \big\}$, and hence of $S_j$, it follows that $x^* + \max(B_j) \in S_j + S_j$ for some $x^* \in S_j \cap [(1- \delta)J^*,J^*]$. Similarly, our assumption on $[J_*,(1+\delta)J_*] \cap \big\{ a_1, \ldots, a_{\delta \ell} \big\}$ implies that $x_* + \min(B_j) \in S_j + S_j$ for some $x_* \in S_j \cap [J_*,(1+\delta)J_*]$, and therefore 
$$\span(S_j + S_j) \, \ge \, \span(B_j) + x^* - x_* \, \ge \, \span(B_j) + \span(J) - \delta(J^* + J_*).$$ 
But by~\eqref{spanJ} we have $\span(J) \ge \sqrt{\delta} \big( J^* + J_* \big)$, so~\eqref{spanBJ} follows.

Now $S_j \subset J$, and so $2 \cdot \span(J) \ge \span(S_j + S_j)$, which, together with~\eqref{spanBJ},  implies that 
\begin{equation}\label{spanBS}
\span(B_j) \, \le \, \frac{1 + \sqrt{\delta}}{2} \span(S_j + S_j) \, = \, \big( 1 + \sqrt{\delta} \big) \span(S_j).
\end{equation}
Next, define the set 
$$A \, = \, \Big\{ a \in S_j \,:\, \big| \big( a + B_j \big) \cap \big( S_j + S_j \big) \big| \ge \big( 1 - \delta \big)|B_j| \Big\}.$$
Note that
$$|S_j \setminus A| \cdot \delta |B_j| \, \le \, \big| \big\{ (a,y) \in S_j \times B_j \,:\, a+y \not\in S_j + S_j \big\} \big| \, \le \, \delta^6 |S_j| \cdot |B_j|,$$
and hence $|A| \ge (1 - \delta^5)|S_j|$. Let $A_* = \min A$ and $A^* = \max A$, and consider the set
$$D \, = \, \big( A_* + B_j \big) \cup \big( A^* + B_j \big).$$

\medskip
\noindent \textbf{Subclaim:} $|D| \ge 3|B_j| / 2$. 

\begin{proof}[Proof of subclaim]
Since $[J_*,(1+\delta)J_*]$ and $[(1- \delta)J^*,J^*]$ each contain more than $\delta^5 \ell$ elements of $\big\{ a_1, \ldots, a_{\delta \ell} \big\}$ and $|A| \ge (1 - \delta^5)|S_j|$, we have 
$$\span(A) \, \ge \, \span(S_j) - \delta\big( J_* + J^* \big) \, \ge \, \big( 1 - 2\sqrt{\delta} \big)\span(S_j),$$
where the second inequality follows by~\eqref{spanJ}, and the fact that $2 \cdot \span(S_j) \ge \span(J)$. Thus, by~\eqref{spanBS},
$$\span(B_j) \,\le\, \big( 1 + \sqrt{\delta} \big) \span(S_j) \,\le\,  \big( 1 + 4\sqrt{\delta} \big) \span(A) \,\le\, 2 \cdot \span(A),$$
and so $\left| \big( A_* + B_j \big) \cap \big( A^* + B_j \big) \right| \le |B_j|/2$, which easily implies the subclaim. 
\end{proof}

Finally, observe that $|D \setminus (S_j + S_j)| \le 2\delta |B_j|$ by the definition of $A$. Hence 
$$\left( \frac{3}{2} - 2 \delta \right) |B_j| \, \le \, |S_j + S_j| \, \le \, |S + S|,$$
as required.
\end{proof}

Now, recall that (by definition) for every set $S$, the intervals $[J_*,(1+\delta)J_*]$ and $[(1- \delta)J^*,J^*]$ each contain at least $\delta^3 \ell$ elements of $S$. Let $\I(S)$ denote the collection of orderings of the elements of $S$ such that the intervals $[J_*,(1+\delta)J_*]$ and $[(1- \delta)J^*,J^*]$ each contain more than $\delta^5 \ell$ elements of $\big\{a_1,\ldots,a_{\delta \ell} \big\}$, and write $\a \in \I$ if $\a \in \I\big( \{a_1,\ldots,a_\ell\} \big)$. Observe that, given a random ordering $\a = (a_1,\ldots,a_\ell)$ of the elements of $S$, the probability that $\a \in \I(S)$ is at least $1/2$. Thus there are at least $\ell! / 2$ orderings $\a \in \I$ for each set $S$.

In order to count sequences $\a \in \I$, recall that $J_0 = \big\{ a \in S : a \le \delta k / \ell \big\}$ and $b = |J_0|$, let $\hat{J}_0 = \big\{ j \in [\ell] : a_j \in J_0 \big\}$, and set
$$Q \, = \, \big\{ 1,\ldots, \delta \ell \big\} \cup \Big\{ j \in [\ell] \,:\, a_j \not\in B_{j-1} \text{ and } j \not\in \hat{J}_0 \Big\}.$$
We claim that if $\a \in \I$, then $|Q| \le 5\delta \ell$. To see this, recall that $r = |S \setminus (J \cup J_0)| \le 3\delta \ell$, and suppose that there are at least $\delta \ell$ values of $j \ge \delta \ell$ with $a_{j+1} \in J \setminus B_j$. Then each such $j$ adds at least $\delta^6 |S_j| \ge \delta^{11} \ell$ new elements to $S_J + S_J$, since $\a \in \I$ implies that $|S_{\delta \ell} | \ge \delta^5 \ell$. But then 
$$|S+S| \, \ge \, |S_J + S_J| \, \ge \, \delta \ell \cdot \delta^{11} \ell \, > \, ck / \ell,$$
since $\ell^3 \ge Ck = k/\delta^{13} > ck / \delta^{12}$, which contradicts our assumption. 

Thus, by the Claim, and setting $q = |Q|$, the number of choices for $S$ is at most
\begin{equation}\label{eq:S+Send}
\sum_{b = 0}^{2\ell/3} \sum_{q = 0}^{5 \delta \ell} \frac{2}{\ell!} {\ell \choose q} {k + q - 1 \choose q - 1}  {\ell \choose b} \left( \frac{\delta  k}{\ell} \right)^{b} \left(  \left( \frac{2}{3} + \delta \right) \frac{ck}{\ell} \right)^{\ell - q - b}.
\end{equation}
Indeed, for each choice of the sets $\hat{J}_0$ and $Q$, and the values of $a_j$ for each $j \in \hat{J}_0 \cup Q$, there are at most $|B_j| \le \left( \frac{2}{3} + \delta \right) |S+S|$ choices for each remaining element $a_{j+1}$. Recall that $q \le 5\delta \ell$, by the observations above,   and that each set $S$ is counted at least $\ell! / 2$ times as a sequence $\a \in \I$. Since $|S+S| \le ck / \ell$,~\eqref {eq:S+Send} follows.

Finally, note that the summand in~\eqref {eq:S+Send} is bounded above by
$$\frac{2}{\ell!} \left( \frac{e \ell}{q} \cdot \frac{2k}{q} \cdot \frac{3\ell}{2ck} \right)^q \left( \frac{e \ell}{b} \cdot \frac{\delta  k}{\ell} \cdot \frac{3\ell}{2ck} \right)^{b} \left(  \left( \frac{2}{3} + \delta \right) \frac{ck}{\ell} \right)^{\ell} \, = \, \frac{2^{O(\delta \ell)}}{\ell!} \left( \frac{3e \ell^2}{cq^2} \right)^q \left( \frac{3 e \delta \ell}{2cb} \right)^{b} \left( \frac{2ck}{3\ell} \right)^{\ell}$$
which is maximized with $b = 3 \delta \ell / 2c$ and $q = 5 \delta \ell$, and so is at most
$$2^{\gamma \ell} \left( \frac{2cek}{3\ell^2} \right)^\ell,$$
where $\gamma = \gamma(\delta,c_0) \to 0$ as $\delta \to 0$ for any fixed $c_0 > 0$. Since we chose $\delta = \delta(c_0)$ to be sufficiently small, the theorem follows.
\end{proof}

The proof of Theorem~\ref{S+S2} is almost identical to the proof of Theorem~\ref{S+S} given above; we need only to add the following observations: that $S_j \subset B_j$, and that $a_{j+1} \not\in S_j$. 

\begin{proof}[Proof of Theorem~\ref{S+S2}]
Let $\delta > 0$, and note that without loss of generality we may assume that $\delta$ is sufficiently small. Suppose that $\ell \in \N$ is sufficiently large and that ${\ell \choose 2} \le k \le \ell^2 / \delta$, and set $c_0 = 2 \ell^2 / k$. Let $C = C(\delta^3,c_0) > 0$ be given by Theorem~\ref{S+S}, and note that since $\delta \le \ell^2 / k \le 2$, $C$ depends only on $\delta$. Let $\lambda \ge 2$ and set $c = \lambda \ell^2 / k$, and note that $\lambda \ell = ck / \ell$ and $\ell^3 \ge (\delta \ell) k \ge C k$. 

We are therefore in the setting of Theorem~\ref{S+S}, and hence we can repeat the proof above up to~\eqref{eq:S+Send}, except replacing $\delta$ everywhere by $\delta^3$. Using the observations that $S_j \subset B_j$ and $a_{j+1} \not\in S_j$, we deduce that the number of choices for $S$ is at most
\begin{equation}\label{eq:T13pf}
\sum_{b = 0}^\ell \sum_{q = 0}^{5\delta^3 \ell} \frac{2}{\ell!} {\ell \choose q} {k + q-1 \choose q-1}  {\ell \choose b} \left( \frac{\delta^3 k}{\ell} \right)^{b} \prod_{j \in Z} \left( \left( \frac{2}{3} + \delta^3 \right) |S+S| - |S_j| \right),
\end{equation}
where $Z = \big\{ j > \delta^3 \ell \,:\, a_{j+1} \in B_j \big\}$, so in particular $|Z| = \ell - q - b$. 

Recalling that $|S_j| \ge j - q - b$ and $|S+S| \le \lambda \ell$, and applying the AM-GM inequality, we see that 
$$\prod_{j \in Z} \left( \left( \frac{2}{3} + \delta^3 \right) |S+S| - |S_j| \right) \, \le \, \left( \left( \frac{2}{3} + \delta^3 \right) \lambda \ell + q + b - \frac{1}{|Z|} \sum_{j \in Z} j \right)^{\ell - q - b},$$
and hence~\eqref{eq:T13pf} is at most
\begin{equation}\label{eq:AMGM}
\sum_{b = 0}^\ell \sum_{q = 0}^{5\delta^3 \ell}  \frac{2^{O(\delta^2 \ell)}}{\ell!} \left( \frac{e^3 k \ell}{q^2} \right)^q \left( \frac{\delta^3 e k}{b} \right)^b  \left( \frac{2\lambda \ell}{3} - \frac{\ell}{2}  + \frac{3(q + b)}{2} \right)^{\ell - q - b}.
\end{equation}
Since ${\ell \choose 2} \le k \le \ell^2 / \delta$ and $\lambda \ge 2$, the summand in~\eqref{eq:AMGM} is maximized with $q = 5 \delta^3 \ell$ and $b \le 3\delta^2 \ell / c$, and hence~\eqref{eq:AMGM} is at most
$$2^{\delta \ell} \left( \frac{(4\lambda - 3)e}{6} \right)^\ell,$$
as required.
\end{proof}

We end this section by briefly discussing the factor $(3/2)^\ell$ which separates the bound in Theorem~\ref{S+S} from that in Conjecture~\ref{S+Sconj}. Although it may seem obvious that we lost this factor in the subclaim, where one might hope that $|D| \ge (2 - \delta)|B_j|$, we remark that this is in fact not the case. Indeed, let $x \ll y \ll z$, and consider the set $S_j = (T + z) \cup (T + 2z)$, where $T = U \cup W$ is composed of $U = [y,y+x]$ and $W$, a random subset of $[0,2y]$ of density $\eps$. Note that if $\delta z > 2y$, then this set is dense near its extremes, and moreover $B_j \approx [z,2y+z] \cup [2z,2y+2z]$ and $A_j \approx (U + z) \cup (U + 2z)$. Thus 
$$A_j + B_j \,\approx\, \big[ y+z, x+3y+z \big] \cup \big[ y+2z, x+3y+2z \big] \cup \big[ y+3z, x+3y+3z \big]$$ 
has size roughly $3 |B_j| / 2$, and so in this case the subclaim is sharp. 

It therefore seems that in order to prove Conjecture~\ref{S+Sconj}, one would have to use some structural properties of $S_j$. Indeed, since $B_j \approx B_{j+1}$ for each $j \in [\ell]$ and a typical $S_j$ is a `random-like' subset of $B_1 \cup \ldots \cup B_{j-1}$, one might hope to bound the probability that (i.e., the number of $S_j$ for which) in the subclaim we have $|D| < (2 - \delta)|B_j|$. However, since the bounds in Theorems~\ref{S+S} and~\ref{S+S2} will suffice to prove Theorems~\ref{CEthm} and~\ref{CEstruc}, we shall not pursue this matter here.

\section{Proof of Theorems~\ref{CEthm} and~\ref{CEstruc}}\label{CESec}

We are now ready to prove Theorem~\ref{CEthm}, which generalizes the Cameron-Erd\H{o}s Conjecture to sum-free sets of size $m$, and its structural analogue, Theorem~\ref{CEstruc}. Both theorems will follow from essentially the same proof; we shall first prove Theorem~\ref{CEthm}, and then point out how the proof can be adapted to deduce Theorem~\ref{CEstruc}. As noted earlier, we shall for simplicity assume throughout that $n$ is even.

The proof is fairly long and technical so, in order to aid the reader, we shall start by giving a brief sketch. The argument is broken into a series of six claims, each relying on the earlier ones; the first five being relatively straightforward, and the last being somewhat more involved. 

We begin, in Claim~1, by using the Hypergeometric Janson Inequality to give a general upper bound on the number of sum-free $m$-sets $I \subset [n]$ with $S = I \cap [n/2]$ fixed. In Claim~2, we use this bound, together with Propositions~\ref{CEprop} and~\ref{prop:Janson}, to prove the theorem in the case $m \ge (1/2 - \delta) n$. Then, in Claim~3, we use Claim~2, Propositions~\ref{CEprop} and~\ref{prop:Janson}, and induction on $n$, to deal with the case $|S| \ge \delta m$. Writing $\ell = |S|$ and $k = \sum_{a \in S} (n/2 - a)$, in Claims~4 and~5 we use Claims~1,~2 and~3 and Lemma~\ref{parts} to deal with the (easy) cases $\ell = O(n/m)$ and $k \gg \ell^2 n / m$. Finally, in Claim~6, we deal with the remaining (harder) cases; however, since we now have $\ell^3 \ge Ck$, we may apply Theorem~\ref{S+S} in place of Lemma~\ref{parts}. In fact, it turns out that when $m = \Theta(n)$ the bound in Theorem~\ref{S+S} is not quite strong enough for our purposes, but in this case we have $|S+S| = O(|S|)$, and so may instead use Theorem~\ref{S+S2}. Each of the claims is stated in such a way as to facilitate the deduction of Theorem~\ref{CEstruc}, which follows by exactly the same argument, with a couple of minor tweaks.

\begin{proof}[Proof of Theorem~\ref{CEthm}]
Fix $\delta > 0$ sufficiently small, let $C = C(\delta) > 0$ be a sufficiently large constant, and let $n \in \N$ and $1 \le m \le n/2$. We shall show that there are at most $2^{C n/m} {n/2 \choose m}$ sum-free $m$-subsets of $[n]$. We shall use induction on $n$, and so we shall assume that the result holds for all smaller values of $n$. Since the result is trivial if $m \le C^{1/3}\sqrt{n}$ (since then $2^{Cn/m} {n/2 \choose m} \ge \binom{n}{m}$), we shall assume that $m \ge C^{1/3} \sqrt{n}$, and thus that $n$ is sufficiently large. 

For any set $I \subset [n]$, let
$$S(I) \,=\, \big\{ x \in I \,:\, x \le n/2 \big\}$$ 
denote the collection of elements of $I$ which are at most $n/2$, as in the statement of Theorem~\ref{CEstruc}. Moreover, given a set $S \subset [n/2]$, let $S' = \big\{ x \in S : x > n/4 \big\}$. 

We begin by giving a general bound on the number of sum-free $m$-sets $I \subset [n]$ with $S(I) = S$, for each $S \subset [n/2]$. For each $\ell \in [m]$ and $k \in \N$, let $\S(k,\ell)$ denote the collection of sets $S \subset [n/2]$ such that $|S| = \ell$ and 
$$\sum_{a \in S} \left( \frac{n}{2} - a \right) \,=\, k.$$
The following claim follows easily from the Hypergeometric Janson Inequality.  

\begin{cclaim}\label{claim:HJI}
For every $k,\ell \in \N$ with $\ell \le m/2$, and every $S \in \S(k,\ell)$, there are at most
$$\min\left\{ {n/2 - |S'+S'| \choose m - \ell}, \, C \cdot \max \left\{  e^{-km^2 / 2n^2}, e^{-km / 8n\ell} \right\} {n/2 \choose m - \ell} \right\}$$
sum-free $m$-sets $I \subset [n]$ such that $S(I) = S$.
\end{cclaim}

\begin{proof}[Proof of Claim~\ref{claim:HJI}]
Since $I$ is sum-free and $S' \subset I$, it follows that $I$ contains no element of $S' + S' \subset \{n/2+1,\ldots,n\}$. Since $S(I) = S$ and $|S| = \ell$, the first bound follows. For the second bound, we use the Hypergeometric Janson Inequality. Define the graph $G$ of `forbidden pairs' by setting
$$V(G) = \big\{ n/2+1, \ldots, n \big\} \quad \text{and} \quad E(G) = \big\{ \{x, x+s\} \colon s \in S \big\},$$
and observe that $I$ is an independent set in $G$, and that $G$ has $k$ edges and maximum degree at most $2\ell$, since $S(I) = S \in \S(k,\ell)$.

Let $\mu$ and $\Delta$ be the quantities defined in the statement of Lemma~\ref{HJI} and note that we are applying the lemma with $|X| = n/2$ and $|R| = m - \ell$. Recalling that $\ell \le m/2$, we have
$$\mu \,=\, k \cdot \frac{(m-\ell)^2}{(n/2)^2} \, \ge \, \frac{km^2}{n^2}  \quad \text{and} \quad \Delta \,\le\, (2\ell)^2 \left( \frac{k}{\ell} \right) \bigg( \frac{m - \ell}{(n/2)} \bigg)^3 \, \le \, \frac{32k\ell (m-\ell)^3}{n^3}.$$
Thus $\mu / 2 \ge km^2 / 2n^2$ and $\mu^2 / 2\Delta \ge km/  8 n \ell$, and so the claimed inequality follows.
\end{proof}

In the calculation below, we shall on several occasions wish to make the assumption that $n - 2m \ge \delta n$. The next claim deals with the complementary case.

\begin{cclaim}\label{almosthalf} 
If $m \ge \left( \ds\frac{1}{2} - \delta \right) n$, then there are at most $\ds\left( \frac{1}{2^{\ell}} + \frac{1}{n^2} \right) {n/2 \choose m}$ sum-free $m$-sets $I \subset [n]$ such that $|S(I)| = \ell$ and $I \not\subset O_n$.
\end{cclaim}

\begin{proof}[Proof of Claim~\ref{almosthalf}]
The result is trivial if $\ell = 0$, so let us assume that $\ell \ge 1$. Recall that $n/2 \ge m \ge C^{1/3} \sqrt{n}$, and hence $m \ge n/3 \ge C^{1/4} \sqrt{n \log n}$,  and that $\delta > 0$ is sufficiently small and $C = C(\delta) > 0$ is sufficiently large. Thus, by Proposition~\ref{CEprop}, there exists $\eps = \eps(\delta) > 0$ such that all but $2^{-\eps m} {n/2 \choose m}$ sum-free $m$-sets $I \subset [n]$ satisfy either $| I \setminus O_n | \le \delta m$, or $| I \setminus B | \le \delta m$ for some interval $B$ of length $n/2$. Moreover, by Proposition~\ref{prop:Janson} there are at most $\frac{1}{n^3} {n/2 \choose m}$ sum-free $m$-sets $I \subset [n]$ such that $1 \le | I \setminus O_n | \le \delta m$. It thus suffices to count sum-free $m$-sets $I \subset [n]$ such that $| I \setminus B | \le \delta m$ for some interval $B$ of length $n/2$.  

We shall divide into two cases, depending on the size of $S(I) \cap [n/4]$.

\bigskip
\noindent \textbf{Case 1:} $\big| S(I) \cap [n/4] \big| \ge \ell / 4$.
\medskip

Suppose $| I \setminus B | \le \delta m$ for some interval $B$ of length $n/2$. If $\big| S(I) \cap [n/4] \big| \ge \ell/4 \ge \delta m$, it follows that $I$ must contain at least $(1 - \delta) m >  \lceil 3n / 8 \rceil$ elements less than $3n/4$, which is impossible since $I$ is sum-free. So assume that $\ell \le 4\delta m$ and note that, by our assumption on $| S(I) \cap [n/4] |$, we have $S(I) \in \S(k,\ell)$ for some $k \ge \ell n / 16$. Thus, by Claim~\ref{claim:HJI}, and setting $\gamma = 10^{-3}$, there are at most
\begin{equation}\label{eqC1}
\sum_{k \ge \ell n / 16} \sum_{S \in \S(k,\ell)} C \cdot \max \left\{ e^{-km^2 / 2n^2}, e^{-km / 8n\ell} \right\} {n/2 \choose m - \ell} \, \le \, n^3 {n/2 \choose \ell} e^{-\gamma n} {n/2 \choose m - \ell}
\end{equation}
sum-free $m$-sets $I \subset [n]$ with $|S(I)| = \ell$, $\big| S(I) \cap [n/4] \big| \ge \ell/4$ and $| I \setminus B | \le \delta m$ for some interval $B$ of length $n/2$. Since $\ell \le 4\delta m$, $\delta > 0$ is sufficiently small, and (trivially) ${n/2 \choose m - \ell} \le {n/2 \choose \ell} {n/2 \choose m}$, this is at most $\frac{1}{n^3} {n/2 \choose m}$, as required.

\bigskip
\noindent \textbf{Case 2:} $|S(I) \cap [n/4]| < \ell/4$.
\medskip

Fix some set $S \subset [n/2]$ with $|S| = \ell$ and such that $S' = S \setminus [n/4]$ has more than $3\ell/4$ elements. Then $|S' + S'| \ge 3\ell/2$ and hence, using~\eqref{trivial2}, it is easy to see that
\begin{equation}\label{eqC2} 
{n/2 - |S'+S'| \choose m - \ell} \, \le \, {n/2 - 3\ell/2 \choose m - \ell} \, = \, {n/2 - 3\ell/2 \choose t - \ell/2} \, \le \, \left( \frac{t}{n/2 - t} \right)^{\ell/2} {n/2 \choose t},
\end{equation}
where $t = n/2 - m \le \delta n$. By Claim~\ref{claim:HJI}, the right-hand side of~\eqref{eqC2} is an upper bound on the number of sum-free $m$-sets $I \subset [n]$ with $S(I) = S$.  

Let us first count sets $I$ such that $\min(I) \ge n/2 - 2\ell$, i.e., such that $a(I) := n/2 - \min( S(I) ) \le 2\ell$. Then there are at most $2^{2\ell}$ choices for the set $S(I)$, and so, by~\eqref{eqC2}, there are at most
$$2^{2\ell} {n/2 - |S'+S'| \choose m - \ell} \, \le \, \left( \frac{16t}{n/2 - t} \right)^{\ell/2} {n/2 \choose t} \, \le \, \frac{1}{2^{\ell+1}} {n/2 \choose m}$$
sum-free $m$-sets $I \subset [n]$ with $|S(I)| = \ell$ and $a(I) \le 2\ell$. 

Now, let us count sets $I$ such that $\min(I) \ge n/2 - 2\ell$, i.e., such that $a = a(I) \ge 2\ell$. Observe that, since $n/2 - a \in I$, then $I$ contains at most $a$ elements of the set $\{n/2+1,\ldots,n/2+a\} \cup \{n-a+1,\ldots,n\}$. Thus $I$ contains at least $m - a - \ell \ge m - 3a/2$ elements of $\{n/2 + a + 1,\ldots,n-a\}$. The remaining elements are contained in a set of size $3a$, and thus, by~\eqref{trivial2}, there are at most
$$2^{3a} {n/2 - 2a \choose m - 3a/2} \, = \,  2^{3a} {n/2 - 2a \choose t - a/2} \, \le \, 2^{3a} \left( \frac{t}{n/2 - t} \right)^{a/2} {n/2 \choose t} \, \le \, \frac{1}{2^{\ell+1}} {n/2 \choose m}$$
sum-free $m$-sets $I$ with $a(I) \ge 2\ell$, where again $t = n/2 - m \le \delta n$. Summing over the various cases, the claim follows.
\end{proof}
 
From now on we shall assume that $n - 2m \ge 2\delta n$. Recall that Claim~\ref{claim:HJI} allows us to count sum-free sets with at most $\delta m$ elements less than $n/2$. We shall use the induction hypothesis to count the sets $I$ that have more than $\delta m$ elements in $[n/2]$.

\begin{cclaim}\label{fewsmall}
There are at most $\delta \cdot 2^{Cn/m} {n/2 \choose m}$ sum-free $m$-sets $I \subset [n]$ with at least $\delta m$ elements less than $n/2$. 
\end{cclaim}

\begin{proof}[Proof of Claim~\ref{fewsmall}]
Recall that $m \ge C^{1/3} \sqrt{n}$ and that $C = C(\delta) > 0$ is sufficiently large. Thus, by Proposition~\ref{CEprop}, there exists $\eps = \eps(\delta) > 0$ such that all but $2^{-\eps m} {n/2 \choose m}$ sum-free $m$-sets $I \subset [n]$ satisfy either $| I \setminus O_n | \le \delta m$, or $| I \setminus B | \le \delta m$ for some interval $B$ of length $n/2$. Moreover, since $\delta > 0$ is sufficiently small, by Proposition~\ref{prop:Janson2} there are at most $2^{Cn/2m} {n/2 \choose m}$ sum-free $m$-sets $I \subset [n]$ with $|I \setminus O_n| \le \delta m$. We may therefore restrict our attention to the collection $\X$ of sum-free $m$-sets $I \subset [n]$ that satisfy $| I \setminus B | \le \delta^3 m$ for some interval $B$ of length $n/2$. 

First, we shall show that there are only few sets in $\X$ which contain more than $\delta^3 m$ elements less than $n/2 - 2\delta^2 n$. Indeed, such a set contains at most $\delta^3 m$ elements of the interval $\{n - 2\delta^2n + 1,\ldots,n\}$ and hence, by the induction hypothesis and~\eqref{trivial2}, there are at most
$$2^{C(n - \delta^2 n)/(m-s)} {n/2 - \delta^2 n \choose m - s} {2\delta^2 n \choose s} \, \le \, 2^{Cn/m} \big( 1 - 2\delta^2 \big)^{m-s} \left( \frac{2m}{n-2m} \right)^s \left( \frac{2e\delta^2 n}{s} \right)^s {n/2 \choose m}$$
such sets with $s \le \delta^3 m$ elements greater than $n - 2\delta^2n$. Summing over $s$, and recalling that $n - 2m \ge \delta n$, it follows that there are at most
\begin{equation}\label{eq:C3}
\sum_{s = 0}^{\delta^3 m} 2^{Cn/m} e^{- 2\delta^2(m-s)} \left( \frac{2m}{\delta n} \cdot \frac{2e\delta^2 n}{s} \right)^s {n/2 \choose m} \, \le \, \delta^3 m \cdot e^{- \delta^2 m} \left( \frac{4e}{\delta^2} \right)^{\delta^3 m} 2^{Cn/m} {n/2 \choose m}
\end{equation}
such sets, which is at most $\delta^2 \cdot 2^{Cn/m} {n/2 \choose m}$, since $\delta > 0$ was chosen sufficiently small and $m \ge C^{1/3}$ is sufficiently large.

It only remains to count the sets in $\X$ which contain at least $\delta m - \delta^3 m > \delta m / 2$ elements of the interval $\{n/2 - 2\delta^2 n, \ldots,n/2\}$, and at most $\delta^3 m$ elements less than $n/2 - 2\delta^2 n$. Note that $m \le 5\delta n$ (else there are no such sets), and so by~\eqref{trivial2} we have 
\begin{equation}\label{C3triv}
{n/2 - \delta m \choose m - a - b} \, \le \, \left( \frac{n/2 - \delta m}{n/2} \right)^{m-a-b} \left( \frac{m}{n/2 - m} \right)^{a+b} {n/2 \choose m} \, \le \,  \left( \frac{3m}{n} \right)^{a+b} {n/2 \choose m}.
\end{equation} 
Now, $|S+S| \ge 2|S| - 1$ for every $S \subset \ZZ$ and so, by Claim~\ref{claim:HJI}, there are at most
$$\sum_{a = 0}^{\delta^3 m} \sum_{b = \delta m / 2}^{m-a} {n/2 \choose a} {2\delta^2 n+1 \choose b} {n/2 - \delta m \choose m - a - b}$$
such sum-free sets. But by~\eqref{C3triv}  this is at most
$$\sum_{a = 0}^{\delta^3 m} \sum_{b = \delta m/2}^{m-a} \left( \frac{3em}{2a} \right)^a \left( \frac{6e\delta^2 m}{b} \right)^b {n/2 \choose m} \, \le \, m^2 \left( \frac{3e}{2\delta^3} \right)^{\delta^3 m} \big( 12e\delta \big)^{\delta m / 2} {n/2 \choose m} \, \le \, e^{-\delta m} {n/2 \choose m},$$
so this completes the proof of the claim.
\end{proof}

From now on, we may restrict our attention to those sum-free subsets $I \subset [n]$ for which $|S(I)| \le \delta^3 m$. 
The remainder of the proof involves some careful counting using Theorems~\ref{S+S} and~\ref{S+S2} and Lemma~\ref{parts}. We shall break up the calculation into three claims. In the first two, which are fairly straightforward, we count the sets $I$ for which $|S(I)|$ is small (Claim~\ref{ellsmall}) or $\sum_{a \in S(I)} (n/2 - a)$ is large (Claim~\ref{kbig}). Finally in Claim~\ref{end}, which is much more delicate, we count the remaining sets.

\begin{cclaim}\label{ellsmall}
If $\ell \le (2C\delta)n/m$, then there are at most $2^{C n/2m} {n/2 \choose m}$ sum-free $m$-sets $I \subset [n]$ with $|S(I)| = \ell$. 
\end{cclaim}

\begin{proof}[Proof of Claim~\ref{ellsmall}]
Let $k \in \N$ and fix a set $S \in \S(k,\ell)$. By Claim~\ref{claim:HJI}, there are at most
\begin{equation}\label{eq:C4}
C \cdot \max \left\{  e^{-km^2 / 2n^2}, e^{-km / 8n\ell} \right\} {n/2 \choose m - \ell}
\end{equation}
sum-free $m$-sets $I \subset [n]$  with $S(I) = S$. Recall that $\ell \le \delta^3 m$ and $n - 2m \ge \delta n$, and suppose first that $e^{-km^2 / 2n^2} \ge e^{-km / 8n\ell}$, i.e., that $\ell \le n/4m$. By Lemma~\ref{parts}, there are at most $\big( \frac{e^2k}{\ell^2} \big)^\ell$ choices for $S$,  and hence, using~\eqref{trivial}, we can bound the number of sets $I$ as follows:
\begin{equation}\label{C4.1}
\sum_k \sum_{S \in \S(k,\ell)} C \cdot e^{-km^2 / 2n^2} {n/2 \choose m - \ell} \, \le \, \sum_{k} C \left( \frac{e^2k}{\ell^2} \right)^\ell e^{-km^2 / 2n^2}  \left( \frac{2m}{n - 2m} \right)^\ell {n/2 \choose m}.
\end{equation}
Now, using~\eqref{eq:kmax} to bound the sum over $k$, this is at most
$$C^2 \left( \frac{e^2}{\ell^2} \cdot \frac{2m}{\delta n} \right)^\ell  \left( \frac{2n^2 \ell}{em^2} \right)^{\ell+1} {n/2 \choose m}\, \le \, \left( \frac{Cn}{m} \right)^3 \left( \frac{4en}{\delta m \ell} \right)^\ell  {n/2 \choose m} \, \le \, 2^{Cn/3m} {n/2 \choose m},$$
where in the last two steps we used the bound $\ell \le n/4m$.

Suppose next that $e^{-km^2 / 2n^2} \le e^{-km / 8n\ell}$, i.e., that $n/4m \le \ell \le (2C\delta) n/m$. The calculation is almost the same:
\begin{eqnarray}\label{C4.2}
\sum_k \sum_{S \in \S(k,\ell)} C \cdot e^{-km / 8n\ell}  {n/2 \choose m - \ell}  & \le & \sum_{k} C \left( \frac{e^2k}{\ell^2} \right)^\ell e^{-km / 8n\ell}  \left( \frac{2m}{n - 2m} \right)^\ell {n/2 \choose m}  \nonumber \\
& \le & \left( \frac{Cn}{m} \right)^3 \left( \frac{4e}{\delta} \right)^\ell  {n/2 \choose m} \, \le \, 2^{Cn/3m} {n/2 \choose m},
\end{eqnarray}
where we again used Lemma~\ref{parts},~\eqref{trivial},~\eqref{eq:kmax} and the bound $\ell \le (2C\delta)n/m$.
\end{proof}

\begin{cclaim}\label{kbig}
If $\ell \ge n/4m$, then there are at most $e^{-\ell} {n/2 \choose m}$ sum-free $m$-sets $I \subset [n]$ such that $S(I) \in \S(k,\ell)$ for some $k \ge \ell^2 n / \delta m$.
\end{cclaim}

\begin{proof}[Proof of Claim~\ref{kbig}]
The calculation is similar to that in the previous claim. Indeed, note that~\eqref{eq:C4} still holds, and that our assumption that $\ell \ge n/4m$ implies that $e^{-km^2 / 2n^2} \le e^{-km / 8n\ell}$. In place of~\eqref{eq:kmax}, we shall use the inequality
\begin{equation}\label{newkmax}
\sum_{k = \ell^2 n / \delta m}^\infty \left( \frac{e^2k}{\ell^2} \cdot \frac{2m}{\delta n} \right)^\ell e^{- km / 8n\ell} \, \le \, \frac{16n\ell}{m} \left( \frac{2e^2}{\delta^2} \right)^\ell e^{- \ell / 8\delta} \, \le \, e^{-\ell},
\end{equation}
which holds since $g(x) = x^a e^{-bx}$ is decreasing on $[a/b,\infty)$ and $g(x + 1/b) < g(x)/2$ if $x > 3a/b$. (Note that we have $\ell^2 n / \delta m > 3\ell \cdot 8n\ell / m$ since $\delta > 0$ is sufficiently small.) Since $|\S(k,\ell)| \le \big( \frac{e^2k}{\ell^2} \big)^\ell$ by Lemma~\ref{parts}, it follows from~\eqref{newkmax} and~\eqref{trivial} that
$$\sum_{k \ge \ell^2 n / \delta m} \sum_{S \in \S(k,\ell)} {n/2 \choose m - \ell} e^{-km / 8n\ell} \, \le \,  \sum_{k \ge \ell^2 n / \delta m}  \left( \frac{e^2k}{\ell^2} \cdot \frac{2m}{\delta n} \right)^\ell e^{-km / 8n\ell} {n/2 \choose m}  \, \le \, e^{-\ell} {n/2 \choose m}.$$
By Claim~\ref{claim:HJI}, this is an upper bound on the number of sum-free $m$-sets $I \subset [n]$  such that $S(I) \in \S(k,\ell)$ for some $k \ge \ell^2 n / \delta m$, as required.
\end{proof}

The following claim now completes the proof of Theorem~\ref{CEthm}. 

\begin{cclaim}\label{end}
If $k \le \ell^2 n / \delta m$ and $\ell \ge (2C\delta) n/m$, then there are at most 
$$\left( 2^{O(\delta \ell )} \left( \frac{3}{2\sqrt{e}} \right)^\ell  + e^{-\delta m} \right) {n/2 \choose m}$$ 
sum-free $m$-sets $I \subset [n]$ with $S(I) \in \S(k,\ell)$. 
\end{cclaim}

This is the most difficult case, and we shall have to count more carefully, using Theorem~\ref{S+S} (in the case $m = o(n)$) and Theorem~\ref{S+S2} (in the case $m = \Theta(n)$). Recall that, given a set $S \subset [n/2]$, we set $S'' = S \cap [n/4]$ and $S' = S \setminus S''$. For simplicity, we shall fix integers $k',\ell' \in \N$ and consider only sets $S \in \S(k,\ell)$ with $|S'| = \ell'$ and 
$$\sum_{a \in S'} \left( \frac{n}{2} - a \right) = k'.$$
Set $\ell'' = \ell - \ell'$ and $k'' = k - k'$, and note that summing over choices of $k'$ and $\ell'$ only costs us a factor of $k\ell = O(\ell^4)$, which is absorbed by the error term $2^{O(\delta\ell )}$. 

\begin{proof}[Proof of Claim~\ref{end}]
We begin by slightly improving the bound in Claim~\ref{claim:HJI}; this will allow us to show that almost every sum-free set $I$ of size $m \ge C\sqrt{n \log n}$ contains no element less than $n/4$. To be precise, we shall show that since $n - 2m \ge \delta n$ and $\ell \le \delta^3 m$, there are at most 
\begin{equation}\label{eq:case5}
C \sum_{S \in \S(k,\ell)} \max \left\{  e^{-\delta \ell'' m^2 / n}, e^{-\delta m} \right\} {n/2 - |S'+S'| \choose m - \ell} \,+\, e^{-\delta m} {n/2 \choose m}
\end{equation}
independent $m$-sets $I \subset [n]$ with $S \in \S(k,\ell)$. 

To prove~\eqref{eq:case5}, we shall partition into two sets by setting
$$\S_1(k,\ell) \, = \, \big\{ S \in \S(k,\ell) \,:\, |S' + S'| \ge \delta n \big\},$$
and $\S_2(k,\ell) = \S(k,\ell) \setminus \S_1(k,\ell)$. By Claim~\ref{claim:HJI}, and using~\eqref{trivial2},  there are at most
$${n/2 \choose \ell} {n/2 - \delta n \choose m - \ell} \, \le \, \big( 1 - 2\delta \big)^{m - \ell} \left( \frac{2m}{\delta n} \right)^\ell \bigg( \frac{en}{2\ell} \bigg)^\ell {n/2 \choose m} \, \le \, e^{-\delta m} {n/2 \choose m},$$
choices for $I$ with $S(I) \in \S_1(k,\ell)$, since $n - 2m \ge \delta n$ and $\ell \le \delta^3 m$. 

Suppose now that $S \in \S_2(k,\ell)$. Similarly as in the proof of Claim~\ref{claim:HJI}, we apply the Hypergeometric Janson Inequality to the graph $G$ with 
$$V(G) = \big\{ n/2+1, \ldots, n \big\} \setminus (S'+S') \quad \text{and} \quad E(G) = \big\{ \{x, x+s\} \colon s \in S'' \big\}.$$
Observe that $k'' \ge \ell'' n / 4$, and hence $G$ has at least 
$$k'' - (2\ell'') |S'+S'| \, \ge \, 2\ell'' \left( \frac{n}{8} - |S'+S'| \right) \, \ge \, \frac{\ell'' n}{8}$$
edges, and maximum degree at most $2\ell''$. Hence, letting $\mu$ and $\Delta$ to be the quantities defined in the statement of Lemma~\ref{HJI}, we have
$$\mu \,\ge\, \frac{\ell'' n}{8} \cdot \frac{(m-\ell)^2}{(n/2)^2} \, \ge \, \frac{\ell'' m^2}{3n}  \quad \text{and} \quad \Delta \,\le\, \frac{n}{2} \cdot \big( 2\ell'' \big)^2 \bigg( \frac{m - \ell}{n/2 - \delta n} \bigg)^3 \le \, \frac{20(\ell'')^2 m^3}{n^2}.$$
Thus $\mu / 2 \ge \delta \ell'' m^2 / n$ and $\mu^2 / 2\Delta \ge \delta m$, and so~\eqref{eq:case5} follows.

In order to complete the calculation, we break into cases according to the order of magnitude of $m$. We begin with the central range.

\bigskip
\noindent \textbf{Case 1:} $C \sqrt{n \log n} \le m \le \delta n$. 
\medskip

For each $c > 0$ let $\S^{(c)}(k,\ell)$ denote the collection of sets $S \in \S(k,\ell)$ with 
$$\frac{ck'}{\ell'} \, \le \, |S'+S'| \, \le \, \frac{(1 + \delta)ck'}{\ell'}.$$
We shall first bound the sum in~\eqref{eq:case5} restricted to $\S^{(c)}(k,\ell)$, for each $c \ge c_0 = \delta^2$, and then sum over choices of $c$. To simplify the calculations, let us also fix $k'$ and $\ell'$, and count only those $S \in \S^{(c)}(k,\ell)$ such that $|S'| = \ell'$ and $\sum_{a \in S'} (n/2 - a) = k'$; as noted above, there are only $O(\ell^4)$ choices for $k'$ and $\ell'$ and this will be absorbed by the error term $2^{O(\delta \ell )}$.

We shall use Theorem~\ref{S+S} to bound the number of sets $S \in \S^{(c)}(k,\ell)$; we may do so since 
$$\frac{(\ell')^3}{k'} \, \ge \, \frac{\ell^3}{2k} \, \ge \, \ell \cdot \frac{\delta m}{2n} \, \ge \, C\delta^2,$$
which follows since $k' \le k \le \ell^2 n / \delta m$ and $(2C\delta) n/m \le \ell \le \delta^3 m$, which together imply that
\begin{equation}\label{eq:ell''}
\ell - \ell' \, = \, \ell''  \, = \, \big| S \cap [n/4] \big| \, \le \, \frac{4k}{n} \, \le \, \frac{4 \ell^2}{\delta m} \, \le \, \delta \ell.
\end{equation}
Thus, by Theorem~\ref{S+S} and~\eqref{trivial2}, it follows that
\begin{equation}\label{case1}
\sum_{S \in \S^{(c)}(k,\ell)} \max \left\{  e^{-\delta \ell'' m^2 / n}, e^{-\delta m} \right\} {n/2 - |S'+S'| \choose m - \ell}
\end{equation}
is at most
\begin{equation}\label{eq:C1.1} 
2^{O(\delta \ell)} \left( \frac{2cek'}{3\ell'^2} \right)^{\ell'} {n/4 \choose \ell''}  \max \left\{  e^{-\delta \ell'' m^2 / n}, e^{-\delta m} \right\} \bigg( \frac{n - 2ck'/\ell'}{n} \bigg)^{m - \ell} \bigg( \frac{2m}{n - 2m} \bigg)^\ell  {n/2 \choose m}.
\end{equation}
Note that we used both the lower bound $|S'+S'| \ge ck' / \ell'$ and the upper bound $|S'+S'| \le (1 + \delta)ck' / \ell'$ from the definition of $\S^{(c)}(k,\ell)$. Since $m \le \delta n$ and $\ell = \ell' + \ell''$, we have
\begin{equation}\label{eq:c13}
\left( \frac{2cek'}{3\ell'^2} \right)^{\ell'} {n/4 \choose \ell''} \bigg( \frac{2m}{n - 2m} \bigg)^\ell \, \le \, 2^{O(\delta \ell)} \left( \frac{4cek'm}{3n\ell'^2} \right)^{\ell'} \bigg( \frac{em}{2\ell''} \bigg)^{\ell''},
\end{equation}
and since $e^{-\delta m^2 / n} \le n^{-C}$ and $\ell'' \le \ell \le \delta^3 m$, we have
\begin{equation}\label{eq:c14}
\bigg( \frac{em}{2\ell''} \bigg)^{\ell''} \max \left\{  e^{-\delta \ell'' m^2 / n}, e^{-\delta m} \right\}  \, \le \, 1.
\end{equation}
Hence~\eqref{case1} is at most
\begin{equation}\label{eq:c12}
2^{O(\delta \ell)} \left( \frac{4cek'm}{3n\ell'^2} \right)^{\ell'} \exp\left( - \frac{2ck'(m - \ell)}{\ell' n} \right) {n/2 \choose m}.
\end{equation}
Finally, using~\eqref{eq:kmax} to sum over $k'$, and summing over $\ell'$, we obtain an upper bound on~\eqref{case1} of 
$$\sum_{\ell'} 2^{O(\delta \ell)} \left( \frac{4cem}{3n\ell'^2} \right)^{\ell'} \left( \frac{\ell'^2 n}{2ce(m - \ell)} \right)^{\ell'+1} {n/2 \choose m} \, \le \, 2^{O(\delta \ell)} \left( \frac{2}{3} \right)^{\ell} {n/2 \choose m},$$
where the error term was able to absorb the extraneous terms because $n/m \le \ell \le \delta^3 m$. 

Applying the above bound on~\eqref{case1} with $c = \delta^2 (1 + \delta)^t$, and summing over integers $0 \le t \le n/\delta^2 m$, we obtain a bound for the sum in~\eqref{eq:case5} over sets $S \in \S(k,\ell)$ such that 
$$\delta^2  \cdot \frac{k'}{\ell'} \, \le \, |S'+S'| \, \le \, \frac{n}{\delta m} \cdot \frac{k'}{\ell'}.$$ 
For those with $|S'+S'| \le \delta^2 k' / \ell'$, the argument above gives an upper bound on~\eqref{case1} of
$$\sum_{\ell'} 2^{O(\delta \ell)} \left( \frac{4\delta^2 ek' m}{3n \ell'^2} \right)^{\ell'} {n/2 \choose m} \, \le \, \big( 4 \delta \big)^\ell {n/2 \choose m}$$
in place of~\eqref{eq:c12}, since we have no lower bound on $|S'+S'|$, and the last inequality follows since $k / \ell^2 \le n/\delta m$. For sets $S \in \S(k,\ell)$ with $|S'+S'| \ge (n / \delta m) \cdot k' / \ell'$, we apply Lemma~\ref{parts} in place of Theorem~\ref{S+S} to obtain, instead of~\eqref{eq:C1.1}, an upper bound of,
$$2^{O(\delta \ell)} \left( \frac{e^2k'}{\ell'^2} \right)^{\ell'} {n/4 \choose \ell''}  \max \left\{  e^{-\delta \ell'' m^2 / n}, e^{-\delta m} \right\} \bigg( 1 - \frac{2 k'}{\delta m \ell'} \bigg)^{m - \ell} \bigg( \frac{2m}{n - 2m} \bigg)^\ell  {n/2 \choose m}.$$
By the same argument as before, this is at most
\begin{equation}\label{eq:S+Sbig}
2^{O(\delta \ell)} \left( \frac{2e^2k'm}{n\ell'^2} \right)^{\ell'} \exp\left( - \frac{2k'(m-\ell)}{\delta m \ell'} \right) {n/2 \choose m} \, \le \, \bigg( \frac{2e^2}{\delta} \bigg)^{\ell'} e^{-\ell' / 3\delta} {n/2 \choose m},
\end{equation}
where the last inequality follows since $k' \le k \le \ell^2 n/\delta m$ (by assumption), $\ell / \ell' \le 2^{O(\delta)}$, $\ell \le \delta^3 m$ and $k' \ge {\ell' \choose 2}$ (otherwise $\S(k',\ell')$ is empty). Summing over $k'$ and $\ell'$, and using the fact that $\ell' \ge (1 - \delta)\ell$, it is easy to see that this is at most $\left( \frac{2}{3} \right)^{\ell} {n/2 \choose m}$, as required.

Putting together the various cases, we see that~\eqref{eq:case5} is bounded above by 
$$2^{O(\delta \ell)} \left( \frac{2}{3} \right)^{\ell} {n/2 \choose m} \,+\, e^{-\delta m} {n/2 \choose m},$$
and since $2/3 < 3 / 2\sqrt{e}$, this proves the claim for $C \sqrt{n \log n} \le m \le \delta n$. 

\medskip
We next observe that the case $m \le C \sqrt{n \log n}$ can be easily reduced to the case above.

\bigskip
\noindent \textbf{Case 2:} $C^{1/3} \sqrt{n} \le m \le C\sqrt{n \log n}$.
\medskip

The proof is the same in Case~1, except for the following step. Instead of the bound $\big( \frac{em}{2\ell''} \big)^{\ell''} \max \big\{  e^{-\delta \ell'' m^2 / n}, e^{-\delta m} \big\} \le 1$, which holds when $m \ge C\sqrt{n \log n}$, we claim that since $m \ge C^{1/3}\sqrt{n}$ we have
\begin{equation}\label{eq:ell''}
\bigg( \frac{em}{2\ell''} \bigg)^{\ell''} \max \left\{  e^{-\delta \ell'' m^2 / n}, e^{-\delta m} \right\} \, \le \, 2^{\delta \ell}.
\end{equation}
Indeed, if $\ell'' \ge n/m$ then the left-hand side of~\eqref{eq:ell''} is at most 1, since $\ell'' \le \delta^3 m$. On the other hand, if $\ell'' \le n/m$ then 
$$\bigg( \frac{em}{2\ell''} \bigg)^{\ell''} \max \left\{  e^{-\delta \ell'' m^2 / n}, e^{-\delta m} \right\} \, = \, \bigg( \frac{em}{2\ell''} \cdot e^{-\delta m^2/n} \bigg)^{\ell''} \le \,  2^{n/2\delta m} \, \le \, 2^{\delta \ell},$$
since $m e^{-\delta m^2/n} \le n/\delta m$, as in the proof of Proposition~\ref{prop:Janson2}, and the function $x \mapsto (c/x)^x$ is maximized when $x = c/e$; the last inequality holds since $\ell \ge (2C\delta) n / m$. The remainder of the calculation is exactly as in Case~1 and so~\eqref{eq:case5} is at most
$$2^{O(\delta \ell)} \left( \frac{2}{3} \right)^{\ell} {n/2 \choose m} \,+\, e^{-\delta m} {n/2 \choose m},$$
as claimed. This completes the proof of the claim for all $C^{1/3} \sqrt{n} \le m \le \delta n$.

\medskip
Finally, we turn to the case $m = \Theta(n)$. We shall assume first that $m \le n/4$, and then (in Case~4) show how the result for $m > n/4$ follows by the same argument.

\bigskip
\noindent \textbf{Case 3:} $\delta n \le m \le n/4$.
\medskip

The calculation in this case is similar to that in Case~1, except we shall use Theorem~\ref{S+S2} in place of Theorem~\ref{S+S}. Indeed, recall  that $C = C(\delta)$ is sufficiently large, and observe that
$$k' \, \le \, k \, \le \, \frac{\ell^2 n}{\delta m} \, \le \, \frac{2}{\delta^2} \cdot \ell'^2 \qquad  \text{and} \qquad \ell' \, \ge \, \frac{(C\delta) n}{m} \, \ge \, 4C \delta,$$
since $4 \le n / m \le 1 / \delta$, recalling that $k \le \ell^2 n / \delta m$, $\ell \ge (2C\delta) n/m$ and $\ell'' \le \delta \ell$ by~\eqref{eq:ell''}. Hence we may apply Theorem~\ref{S+S2} for each $2 \le \lambda \le 2 / \delta^3$, and deduce that there are at most
$$ 2^{\delta \ell'} \left( \frac{(4\lambda - 3)e}{6} \right)^{\ell'}$$
sets $S' \in \S(k',\ell')$ such that $|S'+S'| \le \lambda |S'|$. 

Now, for each $\lambda > 0$ let $\S_{(\lambda)}(k,\ell)$ denote the collection of sets $S \in \S(k,\ell)$ such that
$$\lambda |S'| \,\le \, |S'+S'| \, \le \, \big( 1 + \delta \big) \lambda |S'|.$$ 
As in Case~1, we shall fix $k'$ and $\ell'$ and count only those $S \in \S_{(\lambda)}(k,\ell)$ such that $|S'| = \ell'$ and $\sum_{a \in S'} (n/2 - a) = k'$; once again, there are only $O(\ell^4)$ choices for $k'$ and $\ell'$ and this will be absorbed by the error term $2^{O(\delta \ell )}$.

Now, by Theorem~\ref{S+S2} and~\eqref{trivial2}, 
\begin{equation}\label{case2}
\sum_{S \in \S_{(\lambda)}(k,\ell)} \max \left\{  e^{-\delta \ell'' m^2 / n}, e^{-\delta m} \right\} {n/2 - |S'+S'| \choose m - \ell}
\end{equation}
is at most
$$2^{O(\delta \ell)} \left( \frac{(4\lambda - 3)e}{6} \right)^{\ell'} {n/4 \choose \ell''}  \max \left\{  e^{-\delta \ell'' m^2 / n}, e^{-\delta m} \right\} \bigg( \frac{2m}{n - 2m} \bigg)^\ell \bigg( \frac{n - 2 \lambda \ell'}{n} \bigg)^{m - \ell} {n/2 \choose m}.$$
Note that we used both the lower bound $|S'+S'| \ge \lambda |S'|$ and the upper bound $|S'+S'| \le (1 + \delta) \lambda |S'|$ from the definition of $\S_{(\lambda)}(k,\ell)$. By the same argument as in Case~1 (see~\eqref{eq:c13} and~\eqref{eq:c14}), we have
\begin{equation}\label{eq:c31}
{n/4 \choose \ell''}  \max \left\{  e^{-\delta \ell'' m^2 / n}, e^{-\delta m} \right\} \bigg( \frac{2m}{n - 2m} \bigg)^{\ell''}  \, \le \,  \bigg( \frac{em}{\ell''} \bigg)^{\ell''} \max \left\{  e^{-\delta \ell'' m^2 / n}, e^{-\delta m} \right\}  \, \le \, 1,
\end{equation}
since $m \le n/4$, $e^{-\delta m^2 / n} \le n^{-C}$ and $\ell'' \le \ell \le \delta^3 m$. Thus~\eqref{case2} is at most
$$2^{O(\delta \ell)} \left( \frac{(4\lambda - 3)e}{6} \right)^{\ell'}  \exp\bigg( - \frac{2\lambda \ell' m}{n} \bigg) \bigg( \frac{2m}{n - 2m} \bigg)^{\ell'} {n/2 \choose m},$$
since $\ell = \ell' + \ell''$ and $\ell \le \delta^3 m$ (so the term $e^{O(\lambda \ell^2/n)}$ is absorbed by the error term). By simple calculus\footnote{The function $\frac{(4x - 3) y}{3(1 - 2y)} e^{1-2xy}$ on $[3(1 - \delta),\infty) \times (0,1/4]$ is maximized at $(3(1 - \delta),1/4)$.}, it is straightforward to show that if $\lambda \ge 3(1 - \delta)$ then this is at most 
\begin{equation}\label{eq:c32}
2^{O(\delta \ell)} \bigg( \frac{3}{2\sqrt{e}} \bigg)^\ell {n/2 \choose m},
\end{equation}
as required. 

Applying the bound~\eqref{eq:c32} with $\lambda = 3(1 - \delta)(1+\delta)^t$, and summing over integers $0 \le t \le 1 / \delta^2$, we obtain a bound for the sum in~\eqref{eq:case5} over sets $S \in \S(k,\ell)$ such that 
$$3(1 - \delta)|S'| \, \le \, |S'+S'| \, \le \, \frac{1}{\delta^3} \cdot |S'| \, \le \, (1+\delta)^{1/\delta^2} |S'|.$$
For those with $|S'+S'| \ge |S'| / \delta^3$, we apply Lemma~\ref{parts} to obtain, exactly as in~\eqref{eq:S+Sbig}, a bound of 
$$2^{O(\delta \ell)} \left( \frac{e^2k'}{\ell'^2} \right)^{\ell'} {n/4 \choose \ell''}  \max \left\{  e^{-\delta \ell'' m^2 / n}, e^{-\delta m} \right\} \bigg( 1 - \frac{\ell'}{\delta^3 n} \bigg)^{m - \ell} \bigg( \frac{2m}{n - 2m} \bigg)^\ell {n/2 \choose m}.$$
By~\eqref{eq:c31}, this is at most
$$2^{O(\delta \ell)} \left( \frac{e^2k'}{\ell'^2} \right)^{\ell'} \exp\left( - \frac{\ell'(m-\ell)}{\delta^3 n} \right) {n/2 \choose m} \, \le \, \bigg( \frac{e^2}{\delta^2} \bigg)^{\ell'} \cdot e^{-\ell' / \delta} {n/2 \choose m} \, \le \, \delta^\ell {n/2 \choose m},$$
since $\delta n \le m \le n/4$, $\ell \le \delta^3 m$ and $k' \le \ell^2n / \delta m \le \ell^2 / \delta^2$. 

Finally, for those $S \in \S_{(\lambda)}(k,\ell)$ with $\lambda < 3(1 - \delta)$, we shall need a different weapon. Observe that 
$$|S'+S'| \, \le \, \big( 1 + \delta \big) \lambda |S'| \, \le \, 3|S'| - 4,$$ 
since $3\delta^2 |S'| \ge \delta ^2 \ell > 4$, which holds by~\eqref{eq:ell''} and since $\ell \ge (2C\delta)n/m \ge 8C\delta$. Hence, by Freiman's $3k-4$ Theorem, it follows that $S'$ is contained in an arithmetic progression of length at most 
$$|S'+S'| - |S'| + 1 \, \le \, \big( \lambda - 1 + 3\delta \big) |S'|,$$
which implies that there are at most $k^2 {(\lambda - 1 + 3\delta)\ell \choose \ell}$ sets $S'$ such that $S \in \S_{(\lambda)}(k,\ell)$. Since $|S'+S'| \ge 2\ell' - 1$ and $k = O(\ell^3)$, it follows that~\eqref{case2} is bounded above by
$$2^{O(\delta \ell)} {(\lambda - 1)\ell \choose \ell} \bigg( \frac{2m}{n - 2m} \bigg)^\ell \bigg( \frac{n - 2\lambda\ell}{n} \bigg)^{m}  {n/2 \choose m} \,\le\, 2^{O(\delta \ell)} \bigg( \frac{4}{e^{3/2}} \bigg)^\ell {n/2 \choose m}.$$
The final inequality again follows by simple calculus: the left-hand side is bounded from above by its value with $\lambda = 3$ and $m = n/4$. Since $4 / e^{3/2} < 3 / 2 \sqrt{e}$, the claim follows in this case also.

\medskip
The proof is essentially complete; all that remains is to show that case $m \ge n/4$ can be deduced easily from the case above.

\bigskip
\noindent \textbf{Case 4:} $n/4 \le m \le (1/2 - \delta)n$.
\medskip

We shall reduce this case to the previous one. Indeed, setting $t = n/2 - m$ and noting that $\delta n \le t \le n/4$ and ${n/2 - \lambda \ell \choose m - \ell} = {n/2 - \lambda \ell \choose t - (\lambda - 1)\ell}$, we find that~\eqref{case2} is at most
\begin{equation}\label{case4}
2^{O(\delta \ell)} \left( \frac{(4\lambda - 3)e}{6} \right)^{\ell} \bigg( \frac{2t}{n - 2t} \bigg)^{(\lambda-1)\ell} \bigg( \frac{n - 2 \lambda \ell}{n} \bigg)^t {n/2 \choose t}.
\end{equation}
Since $|S+S| \ge 2|S| - 1$ for every $S \subset \ZZ$, we have $\lambda \ge 2 - \delta$, and the same calculation shows that~\eqref{case4} is at most 
$$2^{O(\delta \ell)} \bigg( \frac{3}{2\sqrt{e}} \bigg)^\ell {n/2 \choose m},$$ 
as required. This completes the proof of Claim~\ref{end}.
\end{proof}

Finally, let us put together the pieces and show that Claims 1--6 prove Theorem~\ref{CEthm}. By Claim~\ref{almosthalf}, there are $O{n/2 \choose m}$ sum-free $m$-subsets of $[n]$ for every $m \ge \big( \frac{1}{2} - \delta \big) n$. By Claim~\ref{fewsmall}, if $m \le \big( \frac{1}{2} - \delta \big) n$ then there are at most $\delta \cdot 2^{Cn/m} {n/2 \choose m}$ such sets with at least $\delta m$ of its elements less than $n/2$. By Claim~\ref{ellsmall}, there are at most $2^{C n/2m} {n/2 \choose m}$ such sets $I$ with $|S(I)| = \ell(I) \le (2C\delta)n/m$, and by Claim~\ref{kbig} there are at most $O{n/2 \choose m}$ such sets such that $k(I) \ge \ell^2 n / \delta m$ and $\ell(I) \ge n/4m$. Finally, by Claim~\ref{end}, there are at most $O{n/2 \choose m}$ such sets which were not contained in any of the previous cases, i.e., such that $\ell(I) \ge (2C\delta)n/m$ and $k(I) \le \ell^2 n / \delta m$. The induction step, and hence Theorem~\ref{CEthm}, now follows.
\end{proof}

We now sketch how the above proof may be adapted in order to prove Theorem~\ref{CEstruc}.

\begin{proof}[Proof of Theorem~\ref{CEstruc}]
Let $\delta > 0$ be a sufficiently small constant, and let $\omega = \omega(n)$ be an arbitrary function such that $\omega \to \infty$ as $n \to \infty$. Let $C > 0$ be sufficiently large, let $m,n \in \N$ satisfy $m \ge C\sqrt{n \log n}$, and consider the sum-free $m$-sets $I \subset [n]$ such that $I \not\subset O_n$. For simplicity, given such a set $I$ let us write $\ell(I) = |S(I)|$, $k(I) = \sum_{a \in S(I)} (n/2 - a)$ and $a(I) = n/2 - \min(S(I))$. 

Suppose first that $m \ge \big( \frac{1}{2} - \delta \big) n$. Then, by the proof of Claim~\ref{almosthalf}, there are $o{n/2 \choose m}$ such sets with $a(I) \ge \sqrt{\omega} = \sqrt{\omega(n)}$, which implies that  $|S(I)| \le \sqrt{\omega}$ and $k(I) \le \omega$ for almost every sum-free $m$-set in $[n]$, as required. Hence we may assume that $m \le \big( \frac{1}{2} - \delta \big) n$.
 
Next, we observe the following strengthening of Claim~\ref{fewsmall} when $m \ge C\sqrt{n \log n}$.

\begin{claim3}
If $m \ge C\sqrt{n \log n}$, then there are $o{n/2 \choose m}$ sum-free subsets $I \subset [n]$ of size $m$ with $I \not\subset O_n$ and at least $\delta m$ elements less than $n/2$. 
\end{claim3}

\begin{proof}[Proof of Claim~3$\,^{\prime}$]
The proof is almost identical to that of Claim~\ref{fewsmall}. The only difference is that when we bound the number of sum-free $m$-sets $I$ such that $1 \le |I \setminus O_n| \le \delta m$, we replace Proposition~\ref{prop:Janson2} by Proposition~\ref{prop:Janson}, which holds for $m \ge C\sqrt{n \log n}$, and implies that there are at most $o{n/2 \choose m}$ such sets. When bounding the size of the collection $\X$ of sum-free $m$-sets $I \subset [n]$ that satisfy $| I \setminus B | \le \delta^3 m$ for some interval $B$ of length $n/2$, we use~\eqref{eq:C3} and note that 
$$m \cdot e^{- \delta^2 m} \left( \frac{4e}{\delta^2} \right)^{\delta^3 m} 2^{Cn/m} {n/2 \choose m} \, \le \, 2^{-\delta^3 m} {n/2 \choose m},$$
since $\delta^3 m > Cn / m$ for $m \ge C\sqrt{n \log n}$. The rest of the proof is exactly the same.
\end{proof}

By Claim~$3^{\prime}$, we may restrict our attention to sum-free $m$-sets $I \subset [n]$ such that $I \not\subset O_n$ and $|S(I)| \le \delta m$. By Claims~\ref{kbig} and~\ref{end}, there are $o{n/2 \choose m}$ such sets with $\ell(I) \ge \frac{Cn}{m} + \omega(n)$. However, if $\ell(I) \le \frac{Cn}{m} + \omega(n)$ and
$$k(I) \, \ge \, \frac{C^4n^3}{m^3} + \omega(n)^4 \, \ge \, \frac{\ell^2 n}{\delta m},$$
then combining the proofs of Claims~\ref{ellsmall} and~\ref{kbig} proves the theorem. Indeed, first note that by Claim~\ref{kbig} there are $o{n/2 \choose m}$ sum-free $m$-sets $I$ with $\ell(I) \ge \max\{ n/4m, \omega\}$. If $\ell(I) \le n/4m$ then, by~\eqref{C4.1} in the proof of Claim~\ref{ellsmall}, there are at most
\begin{equation}\label{eq:1.2.1}
\sum_{k \,\ge\, Cn^3 / m^3 + \omega} C \left( \frac{e^2k}{\ell^2} \right)^\ell e^{-km^2 / 2n^2}  \left( \frac{2m}{n - 2m} \right)^\ell {n/2 \choose m}
\end{equation}
such sets. Now, recall that $n - 2m \ge \delta m$ and note that~\eqref{eq:1.2.1} is decreasing exponentially in $k$. 
Thus if $m = o(n)$, then~\eqref{eq:1.2.1} is at most 
$$\frac{4n^2}{m^2} \cdot \big( C e^{-C} \big)^{n/2m}  {n/2 \choose m} \, = \, o{n/2 \choose m},$$ 
and if $m = \Theta(n)$ then~\eqref{eq:1.2.1} is at most
$$\frac{4n^2}{m^2} \cdot \omega^{\log \omega} e^{-\sqrt{\omega}} {n/2 \choose m} \, = \, o{n/2 \choose m}.$$ 
Finally, if $n/4m \le \ell \le \omega$ then~\eqref{C4.2} and a similar calculation implies there are at most
$$\sum_{k \,\ge\, \omega^4} C \left( \frac{e^2k}{\ell^2} \right)^\ell e^{-km / 8n\ell}  \left( \frac{2m}{n - 2m} \right)^\ell {n/2 \choose m} \, \le \, \omega^{5\omega} e^{-\omega^2/32} {n/2 \choose m} \, = \, o{n/2 \choose m}$$
such sum-free sets $I$, as required. This completes the proof of Theorem~\ref{CEstruc}.
\end{proof}

Finally, we prove that the bounds on $\ell(I) = |S(I)|$ and $k(I) = \sum_{a \in S(I)}(n/2-a)$ given by Theorem~\ref{CEstruc} are best possible up to a constant factor. Indeed, we shall show that if $C\sqrt{n\log n} \le m = o(n)$ and $\eps > 0$ is sufficiently small, then almost every sum-free $m$-subset $I \subset [n]$ satisfies $\ell(I) \ge \eps n / m$ and $k(I) \ge \eps n^3 / m^3$. 

To see that almost every such $I$ satisfies $\ell(I) \ge \eps n / m$, note that by Lemma~\ref{parts} the number of $m$-sets $I \subset [n]$ satisfying $\ell(I) \le \eps n/m$ and $k(I) \le 2Cn^3/m^3$ is at most
\begin{eqnarray*}
 \sum_{\ell = 0}^{\eps n / m} \left( \frac{e^2}{\ell^2} \cdot \frac{2Cn^3}{m^3} \right)^\ell \binom{n/2}{m-\ell} & \le & \sum_{\ell = 0}^{\eps n / m} \left( \frac{e^2}{\ell^2} \cdot \frac{2Cn^3}{m^3} \cdot \frac{3m}{n} \right)^\ell \binom{n/2}{m}\\ 
 & \le & \frac{2\eps n}{m} \left(\frac{6e^2C}{\eps^2}\right)^{\eps n/m} \binom{n/2}{m} \, \le \, 2^{O(\sqrt{\eps}n/m)} \binom{n/2}{m},
\end{eqnarray*}
where the first inequality follows from~\eqref{trivial} and the fact that $m = o(n)$, and the second inequality follows from the fact that $\big( \frac{6e^2Cn^2}{m^2\ell^2} \big)^\ell$ is increasing for $\ell \in (0,\eps n/m]$. On the other hand, by Proposition~\ref{CElower}, there are at least $2^{cn/m}\binom{n/2}{m}$ sum-free $m$-sets in $[n]$, and since $C\sqrt{n\log n} \le m = o(n)$, by Theorem~\ref{CEstruc} almost all of them satisfy $k(I) \le 2Cn^3/m^3$. Thus almost all sum-free $m$-sets $I \subset [n]$ satisfy $\ell(I) > \eps n/m$, as claimed.

To prove that almost every sum-free $m$-subset $I \subset [n]$ satisfies $k(I) \ge \eps n^3 / m^3$ we again apply Lemma~\ref{parts}. Indeed, observe that the number of $m$-subsets $I \subset [n]$ satisfying $\ell(I) \le 2Cn/m$ and $k(I) \le \eps n^3/m^3$ is at most
\begin{eqnarray*}
 \sum_{\ell = 0}^{2C n / m} \left(\frac{e^2}{\ell^2} \cdot \frac{\eps n^3}{m^3} \right)^\ell \binom{n/2}{m-\ell} & \le & \sum_{\ell = 0}^{2C n / m} \left( \frac{e^2}{\ell^2} \cdot \frac{\eps n^3}{m^3} \cdot \frac{3m}{n} \right)^\ell \binom{n/2}{m} \\
 & \le & \frac{3C n}{m} \cdot \exp\left( 2\sqrt{3\eps} \cdot \frac{n}{m} \right) \binom{n/2}{m} \, \le \, 2^{O(\sqrt{\eps}n/m)} \binom{n/2}{m},
\end{eqnarray*}
where the first inequality follows from~\eqref{trivial} and the fact that $m = o(n)$, and the second inequality follows from the fact that $\big( \frac{3\eps e^2 n^2}{\ell^2 m^2} \big)^\ell$ is maximized when $\ell = \sqrt{3\eps} (n/m)$. By Proposition~\ref{CElower} and Theorem~\ref{CEstruc}, there are at least $2^{cn/m}$ sum-free $m$-sets in $[n]$ and almost all of them satisfy $\ell(I) \le 2Cn/m$. It follows that, if $\eps > 0$ is sufficiently small, then almost all sum-free $m$-sets $I \subset [n]$ satisfy $k(I) > \eps n^3/m^3$, as required.

\section*{Acknowledgements}

The third and fourth authors would like to thank Simon Griffiths and Gonzalo Fiz Pontiveros for several useful discussions.

\end{document}